\documentclass[12pt,reqno]{amsart}

\usepackage[T1]{fontenc}
\usepackage{graphicx}
\usepackage{multirow}
\usepackage{amsmath,amssymb,amsfonts}
\usepackage{amsthm,mathtools}
\usepackage{lmodern}
\usepackage{microtype}
\usepackage{mathrsfs}
\usepackage[title]{appendix}
\usepackage{xcolor}
\usepackage{textcomp}
\usepackage{manyfoot}
\usepackage{booktabs}
\usepackage{algorithm}
\usepackage{algorithmicx}
\usepackage{algpseudocode}
\usepackage{listings}
\usepackage[hidelinks]{hyperref}
\hypersetup{pdfauthor={Alexander Panchenko, Ben Hellwig, Oleh Rudenko},pdftitle={Isoperimetric-Combinatorial Bounds for Range-Controlled Matchings and Quasi-Interpolation from Scattered Data},bookmarksnumbered=true}

\theoremstyle{plain}%
\newtheorem{theorem}{Theorem}
\newtheorem{proposition}{Proposition}%
\newtheorem{lemma}{Lemma}
\newtheorem{corollary}{Corollary}

\theoremstyle{remark}%
\newtheorem{example}{Example}%
\newtheorem{remark}{Remark}%

\theoremstyle{definition}%
\newtheorem{definition}{Definition}%
\newtheorem{assumption}{Assumption}

\newcommand{\R}{\mathbb{R}}
\newcommand{\cl}[1]{\overline{#1}}
\newcommand{\Thick}{\mathcal{F}}
\newcommand{\cL}{\mathcal{L}}
\newcommand{\diam}{\operatorname{diam}}
\newcommand{\Qone}{[-1,1]^d}

\newcommand{\Jhat}{\widehat J}
\newcommand{\nbar}{\overline n}
\newcommand{\pbar}{\overline p}

\newcommand{\N}{\mathbb{N}_0}
\newcommand{\Z}{\mathbb{Z}}
\newcommand{\PP}{\mathcal{P}}
\newcommand{\F}{\mathcal{F}}
\newcommand{\one}{\mathbf 1}

\title[Controlled Matchings for Quasi-Interpolation]{Isoperimetric-Combinatorial Bounds for Range-Controlled Matchings and Quasi-Interpolation from Scattered Data}

\author{Alexander Panchenko}
\email{panchenko@wsu.edu}
\address{Department of Mathematics and Statistics, Washington State University, PO Box 643113, Pullman, 99163, USA}

\author{Ben Hellwig}
\email{benjamin.hellwig@wsu.edu}
\address{Department of Mathematics and Statistics, Washington State University, PO Box 643113, Pullman, 99163, USA}

\author{Oleh Rudenko}
\email{oleh.rudenko@wsu.edu}
\address{Department of Mathematics and Statistics, Washington State University, PO Box 643113, Pullman, 99163, USA}

\begin{document}
	
	\begin{abstract}
		
        We develop a mesoscopic framework for analyzing perturbations of finite point sets. 
		Given a reference node set $Y$ with known cubature and approximation properties, we consider a disordered
		node set $Q$ that is observed only through its populations in cubes at scale $r>0$.
		By imposing Hall-type (HT) combinatorial constraints on these populations, we prove the existence
		of a perfect matching between $Y$ and $Q$ with $O(r)$ range.
		This allows integral approximation estimates on coarser cubes at scale $h>r$ to be transferred from $Y$ to $Q$ with explicit error control and anchors
		$Q$ to a periodic grid. We then use translation-invariant quasi-interpolation methods to obtain high-order estimates
		of order $h^s$ as in the quasi-uniform setting, but for a different class of geometries. The key restrictions are the HT conditions and  the bound $r\le Ch$, where $C<1$ is scale independent.

	\end{abstract}



\maketitle

\section{Introduction}

Let
\(U\subset \R^d\) be a bounded domain satisfying the cone and extension hypotheses needed for local polynomial approximation.  In the main construction \(U\) is a cube, or a finite union of axis-aligned cubes. 
We consider a set of distinct reference sites
\[
Y=\{y_1,\dots,y_N\}
\]
chosen for favorable sampling and integration behavior, and an actual disordered point set (sample)
\[
Q=\{q_1,\dots,q_N\}\subset U.
\]
The set $Q$ is produced by an external mechanism and is not known a priori. The distinctive feature of the paper is that we do not assume access to exact point locations of $Q$. 
Instead, the available information is mesoscopic: we assume the populations of $Q$ are known within cubic cells in a partition of $U$, where the cells in this partition have diameter comparable to a prescribed scale $r>0$.

Let
$
\mathcal{C}=\{C_j\}_{j=1}^J,\diam(C_j)\le r,
$
be a  partition of $U $ by a union of axis-aligned cubes $C_j$. The reference populations
$
n_j:=|Y\cap C_j|
$
are known, because $Y$ is known. The observed populations
$p_j:=|Q\cap C_j|$
provide the available description of $Q$. 

The first problem is geometric and combinatorial: find conditions on the population vectors $p, n$ that imply existence of a bijection
\(\tau:Y\to Q\) with \(\|\tau(y)-y\|_\infty=O(r)\).  For empirical measures of equal mass, this is the same as an \(O(r)\) bound for the \(W_\infty\)-distance; see \cite{Villani2009,Santambrogio2015}.

The matching criterion proved in Theorem~\ref{thm:matching} is a Hall-type condition written in terms of the sorted population vectors and a minimal perimeter-cardinality function attached to the cubic partition, or equivalently, to the scaled integer lattice. Thus, pointwise geometric information about \(Q\) is replaced by knowledge of mesoscopic populations. The perimeter-cardinality function measures the smallest number of additional cubes in a one-layer thickening of a set of cubes of given cardinality. The corresponding discrete isoperimetric problems in product grids are related to compression methods and projection inequalities; see \cite{BOLLOBAS199147,veomett2012,loomiswhitney1949,LevRudnev2018}.

Hall's theorem gives the necessary and sufficient conditions for existence of perfect matchings in bipartite graphs \cite{Hall1935}.  Strassen's theorem gives the corresponding coupling criterion for probability measures with a prescribed support relation \cite{Strassen1965}; in Dudley's proof, Strassen's theorem is obtained explicitly from Hall's theorem \cite{Dudley1968}.

Under existing settings, the exact Hall-Strassen conditions must be checked for arbitrary mesocube subsets of $U$, which is exponential in $J$.  In contrast, the HT conditions introduced in this paper replace the full set of Hall inequalities by a smaller family involving the upper partial sums of the components of \(p\) and the lower partial sums of the components of \(n\).  These are only sufficient conditions for matching, but they are explicit and comparatively easy to check. 

The HT conditions resemble discrepancy estimates because they compare total populations over nested increasing families of cubes (see, for example, \cite{Matousek1999,DickPillichshammer2010}).  They also resemble majorization inequalities, since the partial sums of the largest components of \(p\) are compared with the partial sums of the smallest components of \(n\); see \cite{HardyLittlewoodPolya,MarshallOlkinArnold}.  Although more restrictive than the full Hall conditions, our HT conditions are general enough to cover broad classes of point sets that inherit the favorable properties of low-discrepancy sequences and quasi-uniform sets, while being neither low-discrepancy nor quasi-uniform.  

The second problem is approximation from scattered data. Since the set $Q$ is not known pointwise, it is not realistic to base a numerical procedure on values $f(q_i)$ at unspecified sites. A more robust objective is to approximate local integrals of $f$ on coarse cubes, or on slightly larger inflated cubes, and then use these averages to form a piecewise-constant approximation that can be used for constructing a higher-order quasi-interpolant.

Combinatorial bounds in the HT conditions become the main component for developing error estimates. We partition the domain into coarse cubes \(C_\alpha\) of side length \(h>r\), and work locally on the inflated cubes
\[
        \Thick(C_\alpha)=C_\alpha+[-r,r]^d.
\]
On each such cube we approximate integral averages of a given function $f$ by constructing local cubature formulas that are exact for polynomials of degree at most \(s-1\).  When the full set $Q$ is available, these cubatures should use points of $Q$ as nodes.  If only mesoscopic information for $Q$ is known, we prove that $Q$ contains a stable skeleton provided the reference set $Y$ contains a stable cubature skeleton, the HT conditions hold for the pair $Q,Y$, and the scale ratio $r/h$ satisfies a mild smallness assumption. In the present framework, we do not have direct access to $Q$, so the only option is to use the known reference set $Y$ while controlling the resulting error increase. The stability of both cubatures is governed by a polynomial norming margin \(\Gamma_{s-1}\).  Lower bounds for this margin are obtained either from Vandermonde determinants or from metric-span/Remez-type conditions for discrete sets, following \cite{Yomdin2011,BrudnyiYomdin2016}.  

The final approximation operator is a translation-invariant stencil quasi-interpolant on a periodic coarse grid in which exact integral averages of $f$ are approximated by the scattered data cubatures. The polynomial reproduction property of such quasi-interpolants is standard and belongs to the line of spline and shift-invariant quasi-interpolation developed in \cite{DeBoorFix73,LycheSchumaker1975,ChuiDiamond1990,LeiJia1997,Jia2004,Jia2010}.  The approximation scale is \(h\), with disorder scale \(r\) entering through matching and local norming.

The main results are proved in three theorems:
\begin{enumerate}
\item Theorem~\ref{thm:matching} gives the HT population conditions for existence of a range-controlled matching between \(Y\) and \(Q\).  These matching conditions use estimates for a perimeter-cardinality function of the cubic partition.  
\item Theorem~\ref{thm:transfer-reference-cubature}  contains the proof of stability estimates for cubatures. 
\item Theorem~\ref{thm:QI-perturbed-reference} combines cubature replacement and a translation-invariant quasi-interpolation scheme, giving order \(h^s\) estimates under the scale condition \(r\le C h\).  The assumptions do not require global quasi-uniformity.  Local clustering is allowed if it does not prevent polynomial norming.  
\end{enumerate}

The ambient dimension $d$ and the approximation order $s$ are fixed in the asymptotic regime as $N\to\infty$. This is the same as in most work on cubature and quasi-interpolation from scattered data, and we do not attempt to overcome the curse of dimensionality.

The organization of the paper is as follows:  Section~2 defines the local cube geometry and the population data;  Section~3 proves the HT matching criterion and records the perimeter-cardinality estimates needed later;  Section~4 discusses feasible population vectors and examples, including links with low-discrepancy constructions;  Section~5 proves local cubature estimates on inflated cubes;  Sections~6 and 7 combine these estimates with translation-invariant quasi-interpolation and prove the final approximation theorems.

\section{Preliminaries: local cube geometry and population data}\label{sec:prelim-local}

In the combinatorial part (Sections~\ref{sec:matching}--\ref{sec:examples}) we work on a single coarse cube
of side length $h>0$ and a finer mesoscopic partition of side length $r$.

\subsection{A single coarse cube and its mesoscopic partition}\label{subsec:one-cube}
Fix $h>0$ and set
\[
\Omega:=[0,h]^d\subset\mathbb{R}^d.
\]
Fix a mesoscale $r\in(0,h]$ and, for simplicity, assume that $h/r\in\mathbb{N}$.
Partition $\Omega$ into axis-aligned congruent cubes of side length $r$:
\begin{equation}\label{1scale}
\Omega=\bigcup_{j\in I} C_j,\qquad C_j = x_j + [-r/2,r/2]^d,
\end{equation}
where the cubes have pairwise disjoint interiors and $J:=|I|=(h/r)^d$ with boundaries adjusted as needed according to one of the existing conventions to avoid boundary overlap.

\begin{definition}\label{def:thickening}
Define the {\it $r$-thickening} of $\Omega$ by
\[
\widehat\Omega:=\Thick(\Omega):=\Omega+[-r,r]^d.
\]
The partition \eqref{1scale} extends to a partition of $\widehat\Omega$ by congruent cubes of side length $r$; we write $\{\widehat C_j\}_{j\in\hat I}$ for this extended family  and denote
\[
\widehat\Omega=\bigcup_{j\in\hat I}\widehat C_j,
\]
where the cubes have pairwise disjoint interiors, with $\hat I\supset I$,
$\widehat C_j=C_j$ for $j\in I$. We also write $\hat J:=|\hat I|$ for the number of cubes in the thickened partition
(if $h/r\in\mathbb{N}$, then $\hat J=(h/r+2)^d$).
\end{definition}

\subsection{Point sets, populations, and the neighborhood map}\label{subsec:populations}
Let $Q=\{q_1,\dots,q_N\}\subset \Omega$ be the actual (disordered) point set and
let $Y=\{y_1,\dots,y_N\}\subset \widehat\Omega$ be a reference point set, with $|Q|=|Y|=N$.
For $j\in I$ and $j\in\hat I$ define the population counts
\[
p_j:=|Q\cap C_j|,\qquad j\in I,
\qquad\text{and}\qquad
n_j:=|Y\cap \widehat C_j|,\qquad j\in \hat I.
\]
Thus, $\sum_{j\in I} p_j=N$ and $\sum_{j\in\hat I} n_j=N$.

\begin{definition}\label{pointthickening}
For $q_i\in Q$, let $C(q_i)=C_j$ denote its partition cube.
Define the closed neighborhood of $j$ in the thickened partition by
\[
\mathcal N(j):=\{j\}\ \cup\ \bigl\{\,\ell\in \hat I:\ \cl{\widehat C_{\ell}}\cap\cl{C_j}\neq\varnothing\,\bigr\}.
\]
The multi-valued map $\chi:Q\to 2^{Y}$ is defined by
\[
\chi(q_i):= Y\ \cap\ \bigcup_{\ell\in \mathcal N(j)} \widehat C_{\ell},
\qquad\text{with }j\text{ such that }C(q_i)=C_j.
\]
\end{definition} 

Equivalently, $\chi$ defines a bipartite graph $G=(Q,Y,E)$ with edges
\begin{equation}\label{eq:G-def}
(y,q)\in E\qquad\Longleftrightarrow\qquad y\in\chi(q).
\end{equation}

\section{A Hall-type matching criterion from cube populations}\label{sec:matching}

\begin{definition}\label{def:phi}
For $m\in\{0,1,\dots,J\}$ define
\[
\phi(m):=\min\bigl\{\,|\mathcal{N}(I)\setminus I|:\ I\subset\{1,\dots,J\},\ |I|=m\,\bigr\}.
\]
Equivalently, $m+\phi(m)$ is the minimal possible size of a one-step neighborhood $\mathcal{N}(I)$ among all index sets $I$ of size $m$.
\end{definition}

The endpoint identity
\begin{equation}\label{eq:endpoint}
J+\phi(J)=\hat J
\end{equation}
reflects that $\mathcal{N}(\{1,\dots,J\})$ covers all cubes in the partition of $\widehat\Omega$.

\subsection{Hall-type inequalities and a perfect matching of range $2r$}

Consider the bipartite graph $G$ defined by \eqref{eq:G-def}. A perfect matching in $G$ yields a bijection $\tau:Y\to Q$ with $\|y-\tau(y)\|_\infty\le 2r$ for all $y\in Y$.

Let
\[
p_{[1]}\ge p_{[2]}\ge \cdots\ge p_{[J]}
\]
denote the decreasing rearrangement of the vector $p$, and let
\[
n_{(1)}\le n_{(2)}\le \cdots\le n_{(\hat J)}
\]
denote the increasing rearrangement of $n$.
We define the associated partial sums by
\begin{equation}\label{eq:partial-sums-sp-sn}
S_p(m):=\sum_{k=1}^{m} p_{[k]}, \;\;  0\le m\le J, \qquad
S_n(t):=\sum_{j=1}^{t} n_{(j)}, \;\; 0\le t\le \hat J,
\end{equation}
with the convention $S_p(0)=S_n(0)=0$.
Thus, $S_p(m)$ is the sum of the $m$ largest entries of $p$, while $S_n(t)$ is the sum of the $t$ smallest entries of $n$.
These partial sums are standard in majorization theory \cite{MarshallOlkinArnold,HardyLittlewoodPolya}; for non-negative vectors with fixed sum of components (fixed mass), larger upper partial sums correspond to a less uniform distribution.
In particular, a sparse non-negative vector is higher in the majorization order than a more diffuse vector because after decreasing rearrangement, a larger portion of its mass is carried by the first few components.

\begin{theorem}\label{thm:matching}\label{suffcon}
Let $\phi$ be the perimeter-cardinality function from Definition~\ref{def:phi}, and let
\begin{equation}\label{eq:phi-star}
\phi_*(m)\le \phi(m),\qquad m=0,1,\dots,J.
\end{equation}
Assume that the population vectors $n$ and $p$ satisfy the Hall-type (HT) inequalities
\begin{equation}\label{eq:hall-majorization}
S_n(\lceil m+\phi_*(m)\rceil)\ \ge\ S_p(m),\qquad m=0,1,\dots,J.
\end{equation}
In dimensions $d=2,3$, one may take $\phi_* = \phi$, where the exact formulas for $\phi$ are given in Proposition~\ref{prop:exact-profiles}.  In arbitrary dimension $d\ge2$, one may take the explicit lower bound
\begin{equation}
\label{eq:uni-bound}
\phi_*(m)= \bigl(m^{1/d}+2\bigr)^d-m,\qquad 1\le m\le J,
\end{equation}
and $\phi_*(0)=0$. Then the bipartite graph $G$ defined by \eqref{eq:G-def} admits a perfect matching. Thus, there exists a bijection $\tau:Y\to Q$ such that
$\|y-\tau(y)\|_\infty\le 2r$ for all $y\in Y$.
\end{theorem}

\begin{proof}
\noindent
{\it Step 1: Existence of a matching}. 
By Hall's theorem it suffices to verify that for every $S\subset Q$ one has $|N_G(S)|\ge |S|$, where $N_G(S)\subset Y$
is the neighbor set in $G$.
 
Fix $S\subset Q$ and let $I\subset\{1,\dots,J\}$ be the set of indices of cubes intersecting $S$:
\[
I:=\{\,j:\ S\cap C_j\neq\varnothing\,\},\qquad m:=|I|.
\]
Then $|S|\le \sum_{j\in I} p_j\le S_p(m)$ by definition of $S_p$.

Since every $q\in S$ lies in $\bigcup_{j\in I} C_j$, any $y\in Y$ that lies in a cube of the one-step neighborhood $\mathcal{N}(I)$
satisfies $\|y-q\|_\infty\le 2r$ for some $q\in Q$ (because cubes in $\mathcal{N}(I)$ meet the $r$-thickening of $\bigcup_{j\in I} C_j$).
Therefore, $N_G(S)$ contains all points of $Y$ lying in the thickened cubes indexed by $\mathcal{N}(I)$, and hence
\[
|N_G(S)|\ \ge\ \sum_{j\in \mathcal{N}(I)} n_j.
\]
By \eqref{eq:phi-star} and Definition~\ref{def:phi}, $|\mathcal{N}(I)|\ge m+\phi(m)\ge \lceil m+\phi_*(m) \rceil$, and consequently
\[
\sum_{j\in \mathcal{N}(I)} n_j\ \ge\ \min_{\substack{T\subset\{1,\dots,\hat J\}\\ |T|=\lceil m+\phi_*(m)\rceil}}\ \sum_{j\in T} n_j
\ =\ S_n(\lceil m+\phi_*(m)\rceil).
\]
Combining the bounds gives
\[
|N_G(S)|\ \ge\ S_n(\lceil m+\phi_*(m)\rceil)\ \ge\ S_p(m)\ \ge\ |S|,
\]
where the middle inequality is \eqref{eq:hall-majorization}. Hall's theorem now yields a perfect matching.

\noindent{\it Step 2: Universal lower bound}. To prove \eqref{eq:uni-bound}, consider a finite set $A\subset\mathbb{Z}^d$ with $|A|=m>0$. Associate to $A$ the measurable set
\[
U(A):=\bigcup_{a\in A}\bigl(a+[0,1]^d\bigr)\subset\mathbb{R}^d,
\]
a union of unit cubes with pairwise disjoint interiors, so that the Lebesgue measure $\mathcal{L}^d(U(A))=|A|=m$. 
Moreover,
\[
U(A)+[-1,1]^d=U(\Thick(A)),
\]
since for each $a\in\mathbb{Z}^d$ one has
\[
\bigl(a+[0,1]^d\bigr)+[-1,1]^d=a+[-1,2]^d
=\bigcup_{\substack{b\in\mathbb{Z}^d\\ \|b-a\|_\infty\le 1}}\bigl(b+[0,1]^d\bigr).
\]
Therefore
\[
\mathcal{L}^d\bigl(U(A)+[-1,1]^d\bigr)=\mathcal{L}^d\bigl(U(\Thick(A))\bigr)=|\Thick(A)|.
\]
Now apply the Brunn-Minkowski inequality to the measurable sets $U(A)$ and $[-1,1]^d$:
\[
\mathcal{L}^d\bigl(U(A)+[-1,1]^d\bigr)^{1/d}
\ge
\mathcal{L}^d\bigl(U(A)\bigr)^{1/d}+\mathcal{L}^d([-1,1]^d)^{1/d}
= m^{1/d}+2.
\]
Raising both sides to the $d$-th power gives
\[
|\Thick(A)|=\mathcal{L}^d\bigl(U(A)+[-1,1]^d\bigr)\ge (m^{1/d}+2)^d,
\]
which gives \eqref{eq:uni-bound} after subtracting $m$ from both sides.

\end{proof}

\begin{remark}\label{rem:bm-bound}
Since the Lebesgue measures of $U(A),[-1,1]^d$ are positive with $[-1,1]^d$ convex, the Brunn-Minkowski inequality is exact if and only if the two sets in the Minkowski sum are homothetic, up to a set of measure $0$. Hence, the lower bound obtained above is exact for cube sets $S$ of cardinality $m=k^d$, where $k$ is a positive integer. In this case, $\Thick(S)$ is a cube of side $k+2$, and
\[
\phi(k^d)=(k+2)^d-k^d.
\]
Since the complexity of the exact formula for $\phi(m)$ is expected to grow rapidly with dimension, an explicit lower bound $\phi_*(m)$ is useful. 
\end{remark}

\begin{remark}\label{rem: permute}
The HT conditions are permutation invariant with respect to $p$. If some $p$ is feasible for given $n$, so is any $p^\prime$ obtained from $p$ by a permutation of components. 
\end{remark}

\begin{remark}\label{rem:range}
If $\phi_*(0)=0$ and $\phi_*(J)=\hat J-J$ (for example when $\phi_*=\phi$), then the inequalities \eqref{eq:hall-majorization} at $m=0$ and $m=J$ are tautological:
$S_n(0)=S_p(0)=0$ and $S_n(\hat J)=S_p(J)=N$.
Thus, only the indices $m=1,\dots,J-1$ can impose nontrivial constraints.
\end{remark}

\subsection{Exact low-dimensional minimal perimeter}\label{subsec:exact-low-dimensional-profiles}

Theorem \ref{thm:matching} is useful with any lower bound $\phi_*\leq \phi$, but for $d=2,3$, we present the exact minimal perimeter $\phi$ in Proposition \ref{prop:exact-profiles}. The proof uses results from Veomett-Radcliffe~\cite{veomett2012}, and we introduce relevant definitions and notation prior to statement of the proposition.

Denote $\N=\{0,1,2,\dots\}$ and let $d\in\{1,2,3\}$
be the ambient dimension. For integers $a\le b$, we write the \emph{integer
intervals}
\[
[a,b]=\{a,a+1,\dots,b\}\ (\,|[a,b]|=b-a+1\,),
\]
\[
[a,b)=\{a,\dots,b-1\}\ (\,|[a,b)|=b-a\,).
\]
A box $[0,s_1)\times\cdots\times[0,s_d)$ has \emph{side lengths} (point counts)
$s_1,\dots,s_d$. For finite $A\subset\Z^d$ the \emph{thickening} is defined similarly to Definition \ref{def:thickening} as the closed
$\ell^\infty$-neighborhood
\[
\F(A)=\{x\in\Z^d:\ \|x-a\|_\infty\le1\text{ for some }a\in A\},
\]
and the \emph{perimeter} is $\PP(A)=|\F(A)|-|A|$ \cite[\S2]{veomett2012}. Thus,
$\F([0,k])=[-1,k+1]$ in $\Z^1$, and $\PP([0,m))=2$ for all $m\ge1$
(with $\PP(0)=0$). 

We let \emph{the cube order $\prec$} be the Veomett-Radcliffe
well-ordering restricted to $\N^d$ \cite[\S2]{veomett2012} (VR order). On $\N^1$, this is the standard ordering $0<1<2<\cdots$. For $d>1$, letting
$M(u)=\max_j u_j$ and $i_u=\min\{j:u_j=M(u)\}$, one sets $u\prec v$ iff any of
the following three rules holds:
\begin{itemize}
\item R1: $M(u)<M(v)$;
\item R2: $M(u)=M(v)$ and $i_v<i_u$;
\item R3: $M(u)=M(v)$, $i_u=i_v$, and $u'\prec v'$, where $u',v'$ delete
coordinate $i_u$.
\end{itemize}
Let $I_d(m)$ be the length-$m$ initial segment in the VR order. We write $\PP_d(m):=\PP(I_d(m))$
for its perimeter, with the convention
$\PP_d(0)=\PP(\varnothing)=0$. By similar arguments to \cite[Thm.~1]{veomett2012}, initial
segments minimize $\PP$ among all subsets of $\N^d$ of the same cardinality, so in Proposition \ref{prop:exact-profiles} we derive a formula for $\PP_d(m)$, $d \in \{2,3\}$, which is equivalent to $\phi(m)$. 

For a given cardinality $m$, we will decompose $I_d(m)$ into the disjoint union of sets. Let $k$ be such that $k^d \le m < (k+1)^d$ and denote cube $B_d(k):=[0,k)^d$. By R1, $B_d(k) \subseteq I_d(m) \subsetneq B_d(k+1)$. Define \emph{shell} $Z_d(k):=\{x:M(x)=k\}$ so that $B_d(k) \sqcup Z_d(k) = B_d(k+1)$, and using the position of the first maximal
coordinate, decompose the shell into a disjoint union of \emph{faces} $Z_d(k)=\Phi_d(k)\ \sqcup\ \Phi_{d-1}(k)\ \sqcup\cdots\sqcup\ \Phi_1(k)$, where
\[
\Phi_j(k)=\{x:\ x_1,\dots,x_{j-1}<k,\ x_j=k\}=[0,k)^{j-1}\times\{k\}\times[0,k+1)^{d-j},
\]
with $|\Phi_j(k)|=k^{\,j-1}(k+1)^{\,d-j}$. Let $i\in \{0, 1, \dots , d-1\}$ be the index such that $m-k^d$ can be written in the form
\begin{equation}\label{eq:iR}
    m-k^d =\sum_{j=1}^i k^{d-j}(k+1)^{j-1} + R,
\end{equation}
where remainder $R\in [0,k^{d-i-1}(k+1)^i)$. By R2,
\[
B_d(k) \sqcup \Phi_d(k) \sqcup \cdots \sqcup \Phi_{d-i+1}(k) \subseteq I_d(m) \subseteq B_d(k) \sqcup \Phi_d(k) \sqcup \cdots \sqcup \Phi_{d-i}(k).
\]
Because $I_d(m)$ is an initial segment and by R3, $W_d(m):=I_d(m) \cap \Phi_{d-i}(k)$ is a length-$R$ initial segment in dimension $d-1$ under the VR order after deletion of coordinate $d-i$. We decompose
\[
I_d(m) = B_d(k) \sqcup \Phi_d(k) \sqcup \cdots \sqcup \Phi_{d-i+1}  \sqcup W_d(m) \subsetneq B_d(k+1).
\]

\begin{proposition}\label{prop:exact-profiles}
For every $m\ge1$, the minimal perimeter among all subsets of $\N^d$ of
cardinality $m$ is given by 
\[
\PP_d(m)=\PP(\mathfrak B)+\one[R>0]\,\PP_{d-1}(R),
\]
\[
\PP(\mathfrak B)=(k+2)^{\,d-i}(k+3)^{\,i}-k^{\,d-i}(k+1)^{\,i}, 
\]
where $i$ and $R$ are defined by \eqref{eq:iR} and the Iverson bracket $\one[\mathfrak p]$
equals $1$ if proposition $\mathfrak p$ holds and equals $0$ otherwise. In particular
\[
\PP_2(m)=4k+4+2i+2\,\one[R>0]\quad(i\in\{0,1\}),
\]
\[
\PP_3(m)=\PP(\mathfrak B)+\one[R>0]\,\PP_2(R),
\]
where for $d=3$ one has $\PP(\mathfrak B)=6k^2+12k+8,\ 6k^2+16k+12,$ or $6k^2+20k+18$ depending on $i=0,1,2$.

Here, $\mathfrak B = B_d(k) \sqcup \Phi_d(k) \sqcup \cdots \sqcup \Phi_{d-i+1}(k)$ is the box with $i$ side
lengths equal to $k+1$ and $d-i$ side lengths equal to $k$, where $R$ may also be expressed explicitly as $R=m-|\mathfrak B|$.

\end{proposition}

\begin{proof}
Because initial
segments minimize $\PP$ among all subsets of $\N^d$ of the same cardinality, the perimeter of $I_d(m)$ is minimal with respect to $m$.

The cube $B_d(k)$ together with the $i$ completely filled faces is a box
\[
\mathfrak B=[0,k)^{\,d-i}\times[0,k+1)^{\,i}
\]
(each filled face lengthens one axis of the cube from $k$ to $k+1$). Its
thickening is again a box, $\F(\mathfrak B)=[-1,k]^{\,d-i}\times[-1,k+1]^{\,i}$, so the
perimeter is immediate:
\[
\PP(\mathfrak B)=|\F(\mathfrak B)|-|\mathfrak B|=(k+2)^{\,d-i}(k+3)^{\,i}-k^{\,d-i}(k+1)^{\,i}.
\]

It remains to add the perimeter contribution from the last partially filled face, $W_d(m)$. Because $B_d(k) \subseteq \mathfrak B \subseteq \mathfrak B \sqcup W_d(m) \subsetneq B_d(k+1)$, it follows that $W_d(m) \subset \mathcal{F}(\mathfrak B)$, and the only points of $\mathcal{F}(I_d(m)) = \mathcal{F}(\mathfrak B \sqcup W_d(m))$ outside $\mathcal{F}(\mathfrak B)$ are
\[
S:=\mathcal{F}(W_d(m)) \,  \cap \, [-1,k]^{\,d-i-1}\times\{k+1\} \times [-1,k+1]^{\,i},
\]
which is the $(d-1)$-dimensional "slice" of $\mathcal{F}(W_d(m))$ at the level $x_{d-i}=k+1$. Let $W_d(m)^+\subset S$ be the points of $W_d(m)$ with $k+1$ in coordinate $d-i$ instead of $k$. Since deletion of coordinate $d-i$ from $W_d(m)$ gives a length-$R$ initial segment in dimension $d-1$, the same is true for $W_d(m)^+$, which means $|S| = |W_d(m)^+| + P_{d-1}(R) = R + P_{d-1}(R)$. Because $W_d(m) \subset\mathcal{F}(\mathfrak B)$, we must subtract $R = |W_d(m)|$ from the $P(\mathfrak B)$ formula before adding $|S|$ to avoid counting $W_d(m)$ as part of the perimeter for $\mathfrak B \sqcup W_d(m)$. Thus, adding $W_d(m)$ raises the overall perimeter by $\PP_{d-1}(R)$, and
\[
\PP_d(m)=\PP(\mathfrak B)+\one[R>0]\,\PP_{d-1}(R),
\]
which yields the displayed formulas for $d=2,3$.
\end{proof}

\section{Feasibility and examples}\label{sec:examples}

The Hall-type conditions \eqref{eq:hall-majorization} relate the largest populations of $Q$ with the smallest available populations of $Y$. 
If the lower partial sums $S_n(t)$ decrease, then $S_n(\lceil m+\phi_*(m)\rceil)$
decrease as well, which forces upper partial sums $S_p(m)$ to become smaller.
A less uniform distribution of the reference populations $n$ forces the target populations $p$ to be more uniform.
When $n$ is close to uniform, the admissible region for $p$ is comparatively large and allows more concentrated population vectors.
When $n$ is strongly nonuniform, the inequalities permit only those $p$ whose mass is spread more evenly. At an extreme, the most non-uniform feasible $n$  would only allow one feasible $p$, and that $p$ must be as close as possible to uniform.

\begin{example} \label{ex:extreme}
We consider how much non-uniformity in $p$ is permitted for a uniform $n$. Set
\begin{equation}
\label{eq:uniform-n}
\mu:=\bar n=\frac{N}{\widehat J}\in\mathbb N.
\end{equation}
Then the Hall inequalities take the form
\[
S_p(m)\ \le\ \mu\Bigl\lceil \bigl(m^{1/d}+2\bigr)^d\Bigr\rceil,
\qquad 0\le m\le J.
\]
In particular,
\[
p_{(1)}=S_p(1)\le 3^d\mu.
\]
Thus, a single cube $C_j$ may carry as many as $3^d\mu$ points even though $n$ is exactly uniform. If the two largest components are equal with
\[
p_{(1)}=p_{(2)}=c\mu,
\]
then the bound at $m=2$ yields
\[
2c\mu\le \mu(2^{1/d}+2)^d,
\qquad
c\le \frac{(2^{1/d}+2)^d}{2}.
\]
Thus, the two largest components may both exceed the average value $\mu$ by a substantial factor depending only on $d$.  
\end{example}

\begin{example}
In the reverse situation, we consider uniform $p$ with 
$$
p_i=\pbar=\frac{N}{J},
$$
and let $n=(\nbar,\dots, \nbar) + \epsilon$, where $\nbar$ is the sample mean of the vector $n$ (see \eqref{eq:uniform-n}) and the perturbation $||\epsilon||_\infty\leq \Delta$. We determine the largest admissible $\Delta$ under the conditions of Theorem 1.

For a positive integer $k$, set 
$
J=k^d, \Jhat=(k+2)^d,
$ and, for $m=1,\ldots,J$, denote
$
 A_m=(m^{1/d}+2)^d.
$
For simplicity of computations, assume that the integer $\widehat J$ is even.

The "least favorable" choice of the sorted $n$ compatible with a given bound $\Delta$ and satisfying the normalization $\sum_i^{\Jhat} n_i=N$ is the vector whose first $\Jhat/2$ components are equal to $\nbar-\Delta$ while the remaining components are equal to $\nbar+\Delta$. Inserting this into HT conditions \eqref{eq:hall-majorization}
we find
\begin{equation}
        (\nbar-\Delta)A_m-\pbar m\geq 0
        \label{eq:case1}
\end{equation}
for $m\in[J]$ such that
$A_m\leq \Jhat/2$, and 
\begin{equation}
        (\nbar-\Delta)\frac{\Jhat}{2}
        +(\nbar+\Delta)\left(A_m-\frac{\Jhat}{2}\right)
        -\pbar m\geq 0
        \label{eq:case2}
\end{equation}
for $m\in[J]$ such that
$
 A_m\geq \Jhat/2.
$

Set $\Delta=\nbar\delta$. Given $A_m\le \Jhat/2$, solving \eqref{eq:case1} for $\delta$ yields 
\begin{equation}
        \delta\le
        \frac{J A_m-\Jhat m}{J A_m}.
        \label{eq:abbr-lower}
\end{equation}
For $\Jhat/2 \le A_m<\Jhat$ and $m<J$ one obtains from \eqref{eq:case2}
\begin{equation}
        \delta\le
        \frac{J A_m-\Jhat m}{J(\Jhat-A_m)}.
        \label{eq:abbr-upper}
\end{equation}
By Remark \ref{rem:range}, the endpoint $m=J$ gives equality, so we solve for $\delta$ given $m < J$.

By direct calculation, the right hand side of \eqref{eq:abbr-lower} is decreasing in $m$ 
while the right hand side of \eqref{eq:abbr-upper} is increasing in $m$.
Thus, the inequalities tighten for those $m$ for which $A_m$ is closest to $\Jhat/2$. Provided $m\in [1, J]$ is real valued, using calculus shows the largest admissible $\delta_{\max}$ corresponds to
$
        m_0=\left((k+2)2^{-1/d}-2\right)^d,
$
where $(m_0^{1/d}+2)^d=\Jhat/2$. Calculation then yields
\begin{equation}
        \delta_{\max}
        =
        1-
        \left[
        1-\frac{2}{k}\left(2^{1/d}-1\right)
        \right]^d
        =
        \frac{2d}{k}\left(2^{1/d}-1\right)+O(k^{-2}).
\end{equation}
The ceiling and the integer rounding only affect the lower-order terms. Since $k=h/r$, the corresponding absolute perturbation satisfies
\begin{equation}
        \Delta_{\max}
        =
        \left(2d(2^{1/d}-1)+o(1)\right)
        N\left(\frac{r}{h}\right)^{d+1}.
\end{equation}
At least two different regimes may be of interest here. In the first regime $h$ is fixed while the mesoscale $r(N)\to 0$ as $N\to\infty$. At the same time, $r(N)$ must be large enough compared to, say, a scale $hN^{-1/d}$ associated with a cubic lattice packing of $N$ points into the cube of side $\sim h$. 

The second regime is more directly tied to the cubature and quasi-interpolation framework used in subsequent sections. Let the fixed macroscopic cube $\Omega=L[0,1]^d$ be subdivided into cubes $C_\alpha$ of side length $h=h(N)$. Inside each $C_\alpha$ choose a reference set $Y^\alpha$ of cardinality $N_\alpha$, 
assume
$
        \min_\alpha N_\alpha\to\infty,
$
and partition each $C_\alpha$ into mesocubes of side length $r=\theta h$, where $0<\theta\le c<1$ is fixed. After affine rescaling $C_\alpha\to[0,1]^d$, the mesocube partition is fixed. Hence, the HT perturbation calculation gives an admissible component fluctuation of the form
\begin{equation}
\label{eq:regime2}
        \Delta_{\alpha,\max}
        =
        \Gamma(d,\theta,\eta_0)N_\alpha,
        \qquad
        \Gamma(d,\theta,\eta_0)>0,
\end{equation}
where $\eta_0\in(0,1)$ is a fixed margin parameter (see Section \ref{sec:local-cubature} for more details on the choice of $\eta_0$).

Many classical low-discrepancy sequences such as Sobol' sequences, Faure sequences, and Niederreiter-Xing sequences, satisfy the above perturbation bounds.  We recall here the  general theory of digital nets and digital sequences developed in the monograph of Dick and Pillichshammer \cite{DickPillichshammer2010}. A modern quasi-Monte Carlo survey is given by Dick, Kuo, and Sloan \cite{DickKuoSloan2013}.  The classical constructions themselves go back to Sobol' \cite{Sobol1967}, Faure \cite{Faure1982}, and Niederreiter-Xing \cite{NiederreiterXing1995}.

To connect low-discrepancy geometries with the HT conditions, recall that for a point set 
$
        X_N=\{x_1,\ldots,x_N\}\subset [0,1]^d 
$
and the anchored box with upper corner $t$ defined by 
$
        [0,t)=[0,t_1)\times\cdots\times[0,t_d),
$
the star discrepancy of $X_N$ is 
$
        D_N^*(X_N)
        =
        \sup_{t\in[0,1]^d}
        \left|
        \frac{1}{N}|(X_N\cap [0,t)|-t_1\cdots t_d
        \right|.
$
Assume that the reference sets under consideration satisfy the standard logarithmic estimate
\begin{equation}
        D_N^*(X_N)
        \le
        C_*\frac{(\log N)^d}{N}.
        \label{eq:star-bound}
\end{equation}

For an anchored box $B\subset[0,1]^d$, estimate \eqref{eq:star-bound} gives the counting estimate
$
        \left|
        |X_N\cap B|-N|B|
        \right|
        \le
        C_*(\log N)^d .
$
Since a general axis-parallel box $R\subset[0,1]^d$ is a signed sum of at most $2^d$ anchored boxes, 
\begin{equation}
        \left|
        |X_N\cap R|-N|R|
        \right|
        \le
        C_0(d)(\log N)^d,
        \label{eq:box-counting-error}
\end{equation}
where $C_0(d)$ depends only on $d$ and on the discrepancy constant in \eqref{eq:star-bound}. 
Let a fixed finite partition of $[0,1]^d$ into axis-parallel cubes $C_j$ be given. If $n_j$ is the number of points of $X_N$ in the $j$-th sub-cube and $\nbar_j$ is the corresponding uniform value, then \eqref{eq:box-counting-error} gives 
$
        |n_j-\nbar_j|
        \le
        C_0(d)(\log N)^d
        \label{eq:cell-fluctuation}
$
for every $C_j$ in the partition. We also assume that the number of cubes is fixed.

Then a low-discrepancy set $X_N$ is feasible under HT conditions in the first regime if
\begin{equation}
        (\log N)^d\leq C \left(Nr(N)^{d+1}\right),
        \label{eq:ld-compatibility}
\end{equation}
which clearly holds for a range of appropriate rate functions $r(N)$.

For the second regime of interest, the standard star-discrepancy estimate
implies that every mesocube count has absolute fluctuation at most
\begin{equation}
        C_0(d,\theta)(\log N_\alpha)^d .
\end{equation}
Combining this with the local HT upper bound on admissible fluctuations \eqref{eq:regime2} gives
\begin{equation}
        C_0(d,\theta)(\log N_\alpha)^d
        \le
        \Gamma(d,\theta,\eta_0)N_\alpha,
\end{equation}
which holds for all sufficiently large $N$ whenever $\min_\alpha N_\alpha\to\infty$.

If the local populations satisfy $N_\alpha\asymp N(h/L)^d$, the resulting admissible scale is
\[
        h(N)\to0,
        \quad
        Nh(N)^d\to\infty,
        \quad
        r(N)=\theta h(N),\quad 0<\theta\le c<1.
\]
Equivalently,
$
        N^{-1/d}\ll h(N)\ll 1.
$
If, for instance
$$
        h(N)=L N^{-1/d}(\log N)^{q/d},
        \qquad
        r(N)=\theta h(N),
        \qquad q>0,
$$
then $N_\alpha\asymp (\log N)^q$. The HT maximal allowed fluctuation is of order $(\log N)^q$, whereas the local low-discrepancy fluctuation is only $O((\log\log N)^d)$.

\end{example}

\section{Local cubature on inflated cubes}\label{sec:local-cubature}

We next consider local cubature of degree \(s-1\), $s\ge2$, on the inflated cubes
\[
\widehat C_\alpha:=\Thick(C_\alpha)=x^\alpha+[-\widetilde h/2,\widetilde h/2]^d,
\qquad \widetilde h:=h+2r.
\]
Existence and stability of local cubature will be established by bounding a degree-sensitive norming quantity, which can be estimated
either perturbatively from a reference set or directly from explicit local geometric conditions.

\subsection{Local normalization and the norming margin}

For each coarse cube \(C_\alpha=x^\alpha+[-h/2,h/2]^d\), define the affine normalization
\[
\begin{aligned}
S^\alpha(x)&:=\frac{2(x-x^\alpha)}{\widetilde h},\\
S^\alpha(\widehat C_\alpha)&=\Qone.
\end{aligned}
\]
For a finite set \(G^\alpha\subset \widehat C_\alpha\), write
\[
\overline{\overline{G^\alpha}}:=S^\alpha(G^\alpha)\subset\Qone.
\]

\begin{definition}
\label{def:Gamma}
For a finite set \(Z\subset\Qone\), define the {\it norming margin} as
\[
\Gamma_{s-1}(Z):=
\inf\Bigl\{
\max_{z\in Z}|p(z)|:\ p\in \Pi_{s-1}(\mathbb{R}^d),\
\|p\|_{L_\infty(\Qone)}=1
\Bigr\}.
\]
Its reciprocal
\[
N_{s-1}(Z):=\Gamma_{s-1}(Z)^{-1}
\]
is the associated norming constant.
\end{definition}

If \(Z_1\subset Z_2\subset\Qone\), then
$
\Gamma_{s-1}(Z_2)\ge \Gamma_{s-1}(Z_1).
$
Equivalently,
$
N_{s-1}(Z_2)\le N_{s-1}(Z_1).
$

A lower bound on a norming margin $\Gamma_{s-1}(Z)$ implies existence of stable cubature based on $Z$.
\begin{theorem}
\label{thm:local-cubature-gamma}
Fix \(s\in\mathbb{N}\). For a coarse cube \(C_\alpha\), let
\[
Q^\alpha=\{q_{\alpha,1},\dots,q_{\alpha,N_\alpha}\}\subset C_\alpha
\subset \widehat C_\alpha,
\]
and assume that the normalized local set
$
\overline{\overline{Q^\alpha}}:=S^\alpha(Q^\alpha)\subset\Qone
$
satisfies
\begin{equation}\label{eq:gamma-local}
\Gamma_{s-1}(\overline{\overline{Q^\alpha}})\ge \gamma_0>0.
\end{equation}
Then there exist weights
$
\lambda_{\alpha,1},\dots,\lambda_{\alpha,N_\alpha}\in\mathbb{R}
$
such that
\begin{equation}\label{eq:average-exactness}
\frac{1}{\cL^d(\widehat C_\alpha)}\int_{\widehat C_\alpha} p(x)\,dx
=
\sum_{\nu=1}^{N_\alpha}\lambda_{\alpha,\nu}\,p(q_{\alpha,\nu})
\qquad
\text{for all }p\in \Pi_{s-1}(\mathbb{R}^d),
\end{equation}
and
\begin{equation}\label{eq:l1-average-bound}
\sum_{\nu=1}^{N_\alpha}|\lambda_{\alpha,\nu}|
\le \gamma_0^{-1}.
\end{equation}
Equivalently, with
\[
w_{\alpha,\nu}:=\cL^d(\widehat C_\alpha)\lambda_{\alpha,\nu},
\]
one has
\begin{equation}\label{eq:integral-exactness}
\int_{\widehat C_\alpha} p(x)\,dx
=
\sum_{\nu=1}^{N_\alpha}w_{\alpha,\nu}\,p(q_{\alpha,\nu})
\qquad
\text{for all }p\in \Pi_{s-1}(\mathbb{R}^d),
\end{equation}
and
\begin{equation}\label{eq:l1-physical-bound}
\sum_{\nu=1}^{N_\alpha}|w_{\alpha,\nu}|
\le \cL^d(\widehat C_\alpha)\,\gamma_0^{-1}.
\end{equation}
Moreover, for every \(f\in C^s(\widehat C_\alpha)\),
\begin{equation}\label{eq:cubature-error}
\begin{aligned}
&\left|
\frac{1}{\cL^d(\widehat C_\alpha)}\int_{\widehat C_\alpha} f(x)\,dx
-
\sum_{\nu=1}^{N_\alpha}\lambda_{\alpha,\nu}f(q_{\alpha,\nu})
\right|\\
&\qquad\le
C(d,s)\,(1+\gamma_0^{-1})\,\widetilde h^s
\max_{|\beta|=s}\|D^\beta f\|_{L_\infty(\widehat C_\alpha)}.
\end{aligned}
\end{equation}
\end{theorem}

\begin{proof}
Let
\[
E^\alpha:\Pi_{s-1}(\mathbb{R}^d)\to\mathbb{R}^{N_\alpha},
\qquad
E^\alpha(p):=(p(q_{\alpha,1}),\dots,p(q_{\alpha,N_\alpha})).
\]
After affine normalization from \(\widehat C_\alpha\) to \(\Qone\), assumption \eqref{eq:gamma-local} gives
\[
\|E^\alpha(p)\|_{\ell_\infty}
\ge
\gamma_0\,\|p\|_{L_\infty(\widehat C_\alpha)}
\qquad
\text{for all }p\in \Pi_{s-1}(\mathbb{R}^d).
\]
Hence \(E^\alpha\) is injective and its inverse satisfies
$
\|(E^\alpha)^{-1}\|\le \gamma_0^{-1}.
$

Define a linear functional on the image space \(E^\alpha(\Pi_{s-1}(\mathbb{R}^d))\subset\mathbb{R}^{N_\alpha}\) by
\[
\Lambda^\alpha(E^\alpha p):=
\frac{1}{\cL^d(\widehat C_\alpha)}\int_{\widehat C_\alpha} p(x)\,dx.
\]
Then
$
|\Lambda^\alpha(E^\alpha p)|
\le \|p\|_{L_\infty(\widehat C_\alpha)}
\le \gamma_0^{-1}\|E^\alpha(p)\|_{\ell_\infty}.
$
Thus \(\|\Lambda^\alpha\|\le \gamma_0^{-1}\). By the Hahn-Banach theorem, \(\Lambda^\alpha\) extends to a linear functional on
all of \(\mathbb{R}^{N_\alpha}\) with the same norm. Since \((\mathbb{R}^{N_\alpha},\|\cdot\|_{\ell_\infty})^*=\ell_1^{N_\alpha}\),
there exist coefficients \(\lambda_{\alpha,\nu}\) such that
\[
\Lambda^\alpha(u)=\sum_{\nu=1}^{N_\alpha}\lambda_{\alpha,\nu}u_\nu
\qquad\text{for all }u\in\mathbb{R}^{N_\alpha},
\]
and
$
\sum_{\nu=1}^{N_\alpha}|\lambda_{\alpha,\nu}|
\le \gamma_0^{-1}.
$
Applying this to \(u=E^\alpha(p)\) gives \eqref{eq:average-exactness} and \eqref{eq:l1-average-bound}.
The equivalent form \eqref{eq:integral-exactness}--\eqref{eq:l1-physical-bound} follows by multiplying by \(\cL^d(\widehat C_\alpha)\).

For the error estimate, let \(P^\alpha\) be the Taylor polynomial of \(f\) of degree \(s-1\) at the center \(x^\alpha\) of \(C_\alpha\).
Because the cubature is exact on \(\Pi_{s-1}(\mathbb{R}^d)\),
\[
\frac{1}{\cL^d(\widehat C_\alpha)}\int_{\widehat C_\alpha}(f-P^\alpha)(x)\,dx
-
\sum_{\nu=1}^{N_\alpha}\lambda_{\alpha,\nu}(f-P^\alpha)(q_{\alpha,\nu})
\]
coincides with the left-hand side of \eqref{eq:cubature-error}. Hence
\begin{align*}
&\left|
\frac{1}{\cL^d(\widehat C_\alpha)}\int_{\widehat C_\alpha} f(x)\,dx
-
\sum_{\nu=1}^{N_\alpha}\lambda_{\alpha,\nu}f(q_{\alpha,\nu})
\right|\\
&\qquad\le
\left(1+\sum_{\nu=1}^{N_\alpha}|\lambda_{\alpha,\nu}|\right)
\|f-P^\alpha\|_{L_\infty(\widehat C_\alpha)}.
\end{align*}
The standard Taylor remainder estimate on a cube yields
\[
\|f-P^\alpha\|_{L_\infty(\widehat C_\alpha)}
\le
C(d,s)\,\widetilde h^s
\max_{|\beta|=s}\|D^\beta f\|_{L_\infty(\widehat C_\alpha)}.
\]
Using \eqref{eq:l1-average-bound} gives \eqref{eq:cubature-error}.
\end{proof}

\begin{remark}\label{rem:many-node-advantage}
Theorem~\ref{thm:local-cubature-gamma} uses the whole local set \(Q^\alpha\).
No preliminary reduction to exactly \(\dim\Pi_{s-1}(\mathbb{R}^d)\) nodes is required.
This is important for the direct theorem proved later: once the local set is norming, there exists an exact cubature with controlled \(\ell_1\)-norm of the weights although the weights need not be unique.
\end{remark}

\subsection{Lower bounds for $\Gamma_{s-1}$}

We record three sufficient conditions for the lower bound required in
Theorem~\ref{thm:local-cubature-gamma}. The first uses a nondegenerate Vandermonde minor,
the second uses the metric span as defined by Yomdin~\cite{Yomdin2011}, and the third is an explicit finite criterion derived from the second.

Let
\[
M:=\dim\Pi_{s-1}(\mathbb{R}^d)=\binom{s-1+d}{d}.
\]
For a set \(Z'=\{z_1,\dots,z_M\}\subset\Qone\), define the Vandermonde determinant
\[
\Delta_{s-1}(Z'):=\det\bigl(z_i^{\beta_j}\bigr)_{1\le i,j\le M}\,,
\]
where \(\{x^{\beta_j}:|\beta_j|\le s-1\}\) is the standard monomial basis on \(\Qone\), arranged in a fixed order.

\begin{proposition}\label{prop:det-gamma}
Let \(Z\subset\Qone\) be finite. Assume that \(Z\) contains a subset \(Z'\) of cardinality \(M\) such that
\[
|\Delta_{s-1}(Z')|\ge \eta_0>0.
\]
Then
\begin{equation}\label{eq:gamma-det}
\Gamma_{s-1}(Z)\ge \Gamma_{s-1}(Z')\ge \frac{\eta_0}{M\,M!}.
\end{equation}
\end{proposition}

\begin{proof}
It is enough to consider  a subset \(Z'\) of cardinality $M$.
Proposition~2.1 and Corollary~2.2 of Brudnyi-Yomdin
\cite{BrudnyiYomdin2016} give
\[
N_{s-1}(Z')\le \frac{M\,M!}{|\Delta_{s-1}(Z')|},
\]
and hence \eqref{eq:gamma-det} after taking reciprocals.
\end{proof}

For \(Z\subset\Qone\), let \(M(Z,\varepsilon)\) denote the covering number in the \(\ell_\infty\)-metric.
Let \(M_{s-1,d}(\varepsilon)\) be the corresponding Vitushkin polynomial--a polynomial in \(\varepsilon^{-1}\) of degree at most \(d-1\) with coefficients depending only on \(d\) and \(s\); see \cite{Yomdin2011,BrudnyiYomdin2016}. We note that Vitushkin polynomials are independent of the choice of the point set
$Z$.

Following Yomdin \cite{Yomdin2011}, define the metric span by
\begin{equation}\label{eq:metric-span-def}
\omega_{s-1,d}(Z):=
\sup_{\varepsilon>0}\varepsilon^d\bigl(M(Z,\varepsilon)-M_{s-1,d}(\varepsilon)\bigr).
\end{equation}

\begin{proposition}\label{prop:metric-span-gamma}
Let \(Z\subset\Qone\) be finite or compact, and suppose that
\[
\omega_{s-1,d}(Z)\ge \omega_0>0.
\]
Then
\begin{equation}\label{eq:gamma-metric}
\Gamma_{s-1}(Z)\ge R_{s-1,d}(\omega_0)^{-1},
\end{equation}
where
\begin{equation}\label{eq:remez-constant}
R_{s-1,d}(\omega):=
T_{s-1}\!\left(
\frac{1+(1-\omega)^{1/d}}{1-(1-\omega)^{1/d}}
\right),
\end{equation}
and \(T_{s-1}\) is the Chebyshev polynomial of degree \(s-1\).
\end{proposition}

\begin{proof}
This is given by Theorem~2.5 in Brudnyi-Yomdin \cite{BrudnyiYomdin2016}.
\end{proof}

For a finite set \(Z\subset\Qone\), define
\[
\operatorname{sep}_\infty(Z):=\min\{\|z-z'\|_\infty:\ z,z'\in Z,\ z\ne z'\}.
\]

\begin{corollary}\label{cor:finite-criterion}
Let \(Z\subset\Qone\) be finite and denote \(\varepsilon_0:=\operatorname{sep}_\infty(Z)\).
If
\begin{equation}\label{eq:finite-vitu}
|Z| > M_{s-1,d}(\varepsilon_0),
\end{equation}
then \(\omega_{s-1,d}(Z)>0\), and therefore
\[
\Gamma_{s-1}(Z)
\ge R_{s-1,d}(\omega_{s-1,d}(Z))^{-1}>0.
\]
In particular, Theorem~\ref{thm:local-cubature-gamma} applies to \(Z\).
\end{corollary}

\begin{proof}
By Proposition~3.5 of Yomdin \cite{Yomdin2011}, condition \eqref{eq:finite-vitu} implies that \(\omega_{s-1,d}(Z)>0\).
Then Proposition~\ref{prop:metric-span-gamma} yields the stated lower bound.
\end{proof}

\begin{remark}
\label{rem:finite-criterion-sparse}
The powers of \(\varepsilon_0\)
in \(M_{s-1,d}(\varepsilon_0)\) given in the above Corollary~\ref{cor:finite-criterion} are less restrictive than the cardinality of a cubic grid
\(\varepsilon_0\)-packing of \(\Qone\).
Indeed, the threshold in \eqref{eq:finite-vitu} has order at most \(\varepsilon_0^{-(d-1)}\),
whereas a periodic cubic packing with separation comparable to \(\varepsilon_0\) has cardinality of order \(\varepsilon_0^{-d}\).
Consequently, the criterion applies to families that are sparser than a full periodic packing.
For example, any family satisfying \(|Z|\gtrsim \varepsilon_0^{-q}\) for some \(q>d-1\) will satisfy
\eqref{eq:finite-vitu} once \(\varepsilon_0\) is sufficiently small.
In particular, the class of sets covered by the corollary is nonempty in the cube $[-1,1]^d$:
one may choose a separation scale \(\varepsilon_*(d,q)>0\), depending only on \(d\) and \(q\),
for which there exist such sets \(Z\subset\Qone\).
After rescaling to a cube of side length \(h\), this produces a reference separation of order $h\varepsilon_*(d,q)$.
\end{remark}

\subsection{Perturbative transfer from a reference set}
The following lemma allows transfer of a lower bound for \(\Gamma_{s-1}\)
from a reference set to a corresponding matched set with controllable loss.

\begin{lemma}
\label{lem:gamma-perturb}
There exists a constant \(\mathfrak{M}(d,s)>0\), depending only on \(d\) and \(s\), such that the following holds:
Let \(Y,Q\subset\Qone\) be finite sets of the same cardinality, and assume that there is a bijection
\[
\tau:Y\to Q
\]
with
\[
\max_{y\in Y}\|\tau(y)-y\|_\infty\le \delta.
\]
Then
\begin{equation}\label{eq:gamma-perturb}
\Gamma_{s-1}(Q)\ge \Gamma_{s-1}(Y)-\mathfrak{M}(d,s)\,\delta.
\end{equation}
In particular, if
\[
\delta\le \frac{\Gamma_{s-1}(Y)}{2\mathfrak{M}(d,s)},
\]
then
\[
\Gamma_{s-1}(Q)\ge \tfrac12\Gamma_{s-1}(Y).
\]
\end{lemma}

\begin{proof}
Let \(p\in \Pi_{s-1}(\mathbb{R}^d)\) satisfy \(\|p\|_{L_\infty(\Qone)}=1\).
By the Markov brothers' inequality on the cube,
\[
\|\nabla p\|_{L_\infty(\Qone)}\le \mathfrak{M}(d,s).
\]
Hence, for every \(y\in Y\),
$
|p(\tau(y))-p(y)|\le \mathfrak{M}(d,s)\,\delta,
$
so
$
\max_{q\in Q}|p(q)|\ge \max_{y\in Y}|p(y)|-\mathfrak{M}(d,s)\,\delta.
$
Taking the infimum over all such \(p\) gives \eqref{eq:gamma-perturb}.
\end{proof}

To prove the main perturbative result of this section, let standard monomial basis
\[
\Psi=(x^{\beta_1},\dots,x^{\beta_M})
\qquad (|\beta_j|\le s-1)
\]
on \(\Qone\) be listed in one fixed order, and let
\[
b_j:=2^{-d}\int_{\Qone}x^{\beta_j}\,dx,
\qquad
b=(b_1,\dots,b_M)^T.
\]

\begin{theorem}
\label{thm:transfer-reference-cubature}
For coarse cube \(C_\alpha\) with side length $h>0$, suppose that:
\begin{enumerate}
\item there is a reference skeleton
\[
(Y^\alpha)_{\mathrm{sk}}=\{y_{\alpha,1},\dots,y_{\alpha,M}\}\subset Y^\alpha;
\]
\item the normalized skeleton
\[
\overline{\overline{Y^\alpha}}_{\mathrm{sk}}:=S^\alpha((Y^\alpha)_{\mathrm{sk}})=\{\overline{\overline{y}}_{\alpha,1},\dots,\overline{\overline{y}}_{\alpha,M}\}\subset \Qone
\]
satisfies
\[
\bigl|\Delta_{s-1}(\overline{\overline{Y^\alpha}}_{\mathrm{sk}})\bigr|\ge \eta_0>0;
\]
\item there is a bijection \(\tau^\alpha:Y^\alpha\to Q^\alpha\) such that
\[
\|\tau^\alpha(y)-y\|_\infty\le 2r
\qquad\text{for all }y\in Y^\alpha;
\]
\item the ratio $r/h$ satisfies scale separation condition
\begin{equation}\label{eq:scale-separation}
\frac{r}{h}\le \frac{1}{\frac{4MM!(s-1)}{\eta_0\kappa}-2},
\end{equation}
where $\kappa\in (0,1)$ is fixed such that $\kappa < 2MM!(s-1)/\eta_0$.
\end{enumerate}
Define the matched skeleton on the physical sample by
\[
(Q^\alpha)_{\mathrm{sk}}:=\{q_{\alpha,1},\dots,q_{\alpha,M}\},
\qquad
q_{\alpha,\nu}:=\tau^\alpha(y_{\alpha,\nu}),
\]
and write
\[
\overline{\overline{q}}_{\alpha,\nu}:=S^\alpha(q_{\alpha,\nu}),
\qquad
\overline{\overline{Q^\alpha}}_{\mathrm{sk}}:=\{\overline{\overline{q}}_{\alpha,1},\dots, \overline{\overline{q}}_{\alpha,M}\}.
\]
Let
\[
V^{\alpha,Y}:=\bigl(\overline{\overline{y}}_{\alpha,\nu}^{\beta_j}\bigr)_{1\le \nu,j\le M}\,,
\qquad
V^{\alpha,Q}:=\bigl(\overline{\overline{q}}_{\alpha,\nu}^{\beta_j}\bigr)_{1\le \nu,j\le M}
\]
be the normalized Vandermonde matrices. Then the following conclusions hold:
\begin{enumerate}
\item Both matrices are invertible. The coefficient vectors
\[
    \lambda^{\alpha,Y}:=(V^{\alpha,Y})^{-T}b,
    \qquad
    \lambda^{\alpha,Q}:=(V^{\alpha,Q})^{-T}b
    \]
    define cubature rules
    \[
    \widetilde A^{\alpha,Y}(f):=\sum_{\nu=1}^M\lambda_\nu^{\alpha,Y}f(y_{\alpha,\nu}),
    \quad
    \widetilde A^{\alpha,Q}(f):=\sum_{\nu=1}^M\lambda_\nu^{\alpha,Q}f(q_{\alpha,\nu})
    \]
which are exact on \(\Pi_{s-1}(\mathbb{R}^d)\) for the inflated average
\[
A^\alpha(f):=\frac{1}{\cL^d(\widehat C_\alpha)}\int_{\widehat C_\alpha}f(x)\,dx.
\]
Moreover,
\begin{equation}\label{eq:Y-weight-l1}
\|\lambda^{\alpha,Y}\|_{\ell_1}\le \frac{M\,M!}{\eta_0},
\qquad
\|\lambda^{\alpha,Q}\|_{\ell_1}\le \frac{M\,M!}{\eta_0(1-\kappa)}.
\end{equation}
\item The cubature weights satisfy the perturbative estimate
\begin{equation}\label{eq:YQ-weight-difference}
\|\lambda^{\alpha,Q}-\lambda^{\alpha,Y}\|_{\ell_1}
\le
M^2(s-1)\frac{(M!)^2}{\eta_0^2(1-\kappa)}\frac{4r}{\widetilde h}.
\end{equation}
\item For every \(f\in C^s(\widehat C_\alpha)\),
\begin{equation}\label{eq:Y-cubature-error}
\left|A^\alpha(f)-\widetilde A^{\alpha,Y}(f)\right|
\le
C(d,s,\eta_0)\,\widetilde h^s
\max_{|\beta|=s}\|D^\beta f\|_{L_\infty(\widehat C_\alpha)},
\end{equation}
\begin{equation}\label{eq:Q-cubature-error}
\left|A^\alpha(f)-\widetilde A^{\alpha,Q}(f)\right|
\le
C(d,s,\eta_0)\,\widetilde h^s
\max_{|\beta|=s}\|D^\beta f\|_{L_\infty(\widehat C_\alpha)},
\end{equation}
and
\begin{equation}\label{eq:YQ-cubature-difference}
\left|\widetilde A^{\alpha,Y}(f)-\widetilde A^{\alpha,Q}(f)\right|
\le
C(d,s,\eta_0)\,r\,\widetilde h^{s-1}
\max_{|\beta|=s}\|D^\beta f\|_{L_\infty(\widehat C_\alpha)}.
\end{equation}
\end{enumerate}
\end{theorem}

\begin{proof}
Rearranging \eqref{eq:scale-separation} gives
\begin{equation}\label{eq:small-r-condition}
M(s-1)\frac{M!}{\eta_0}\frac{4r}{\widetilde h}\le \kappa,
\end{equation}
where $\widetilde{h} = h+2r$. Set
$
\delta_\alpha:=\max_{1\le \nu\le M}\|\overline{\overline{q}}_{\alpha,\nu}-\overline{\overline{y}}_{\alpha,\nu}\|_\infty.
$
Since \(S^\alpha\) rescales distances by the factor \(2/\widetilde h\), the matching bound gives
\begin{equation}\label{eq:delta-alpha-bound}
\delta_\alpha\le \frac{4r}{\widetilde h}.
\end{equation}

Every cofactor of \(V^{\alpha,Y}\) has absolute value at most \((M-1)!\) since \(|\overline{\overline{y}}_{\alpha,\nu}^{\beta_j}|\le1\). Using this in the formula for the inverse matrix we find
\begin{equation}\label{eq:inverse-VY-bound}
\|(V^{\alpha,Y})^{-1}\|_{\infty}
\le \frac{M!}{\eta_0}.
\end{equation}
Next, for each monomial \(x^{\beta_j}\) and each pair \(u,v\in \Qone\), the mean-value theorem gives
\[
|u^{\beta_j}-v^{\beta_j}|
\le |\beta_j|\,\|u-v\|_\infty
\le (s-1)\|u-v\|_\infty.
\]
Therefore
\begin{equation}\label{eq:VQ-VY-bound}
\|V^{\alpha,Q}-V^{\alpha,Y}\|_{\infty}
\le M(s-1)\,\delta_\alpha.
\end{equation}
Let
\[
E^\alpha:=(V^{\alpha,Y})^{-1}(V^{\alpha,Q}-V^{\alpha,Y}).
\]
By \eqref{eq:inverse-VY-bound}, \eqref{eq:VQ-VY-bound}, \eqref{eq:delta-alpha-bound}, and \eqref{eq:small-r-condition}
\begin{equation}\label{eq:E-alpha-bound}
\|E^\alpha\|_{\infty}
\le M(s-1)\frac{M!}{\eta_0}\frac{4r}{\widetilde h}
\le \kappa <1.
\end{equation}
Consequently \(I+E^\alpha\) is invertible by the Neumann series, and
\[
V^{\alpha,Q}=V^{\alpha,Y}(I+E^\alpha)
\]
is invertible with
\[
\|(V^{\alpha,Q})^{-1}\|_{\infty}
\le \frac{\|(V^{\alpha,Y})^{-1}\|_{\infty}}{1-\|E^\alpha\|_{\infty}}
\le \frac{M!}{\eta_0(1-\kappa)}.
\]
Since \(|b_j|\le1\) for every \(j\), one has \(\|b\|_{\ell_1}\le M\), and therefore
\[
\|\lambda^{\alpha,Y}\|_{\ell_1}
\le \|(V^{\alpha,Y})^{-T}\|_{\ell_1\to\ell_1}\,\|b\|_{\ell_1}
\le \frac{M\,M!}{\eta_0},
\]
\[
\|\lambda^{\alpha,Q}\|_{\ell_1}
\le \|(V^{\alpha,Q})^{-T}\|_{\ell_1\to\ell_1}\,\|b\|_{\ell_1}
\le \frac{M\,M!}{\eta_0(1-\kappa)},
\]
which proves \eqref{eq:Y-weight-l1}. The identities
\[
(V^{\alpha,Y})^T\lambda^{\alpha,Y}=b,
\qquad
(V^{\alpha,Q})^T\lambda^{\alpha,Q}=b
\]
mean \(\widetilde A^{\alpha,Y}\) and \(\widetilde A^{\alpha,Q}\) integrate the normalized monomial basis in the same way as the
inflated average. Hence, both rules are exact on \(\Pi_{s-1}(\mathbb{R}^d)\).

For the difference of the weights, write
\[
V^{\alpha,Q}=V^{\alpha,Y}(I+E^\alpha),
\qquad
(V^{\alpha,Q})^{-T}=(V^{\alpha,Y})^{-T}(I+E^\alpha)^{-T}.
\]
Since \(\|E^\alpha\|_\infty\le \kappa\), the Neumann series also gives
\[
\|(I+E^\alpha)^{-T}-I\|_{\ell_1\to\ell_1}
\le \frac{\|(E^\alpha)^T\|_1}{1-\|(E^\alpha)^T\|_1}
\le \frac{\|E^\alpha\|_\infty}{1-\kappa}.
\]
Then \eqref{eq:YQ-weight-difference} follows from
\[
\begin{aligned}
\|\lambda^{\alpha,Q}-\lambda^{\alpha,Y}\|_{\ell_1}
&\le
\|(V^{\alpha,Y})^{-T}\|_{\ell_1\to\ell_1}
\,\|(I+E^\alpha)^{-T}-I\|_{\ell_1\to\ell_1}
\,\|b\|_{\ell_1}\\
&\le
\frac{M!}{\eta_0} \cdot \frac{\|E^\alpha\|_\infty}{1-\kappa}\cdot M\\
&\le
M^2(s-1)\frac{(M!)^2}{\eta_0^2(1-\kappa)}\frac{4r}{\widetilde h},
\end{aligned}
\]
using \eqref{eq:E-alpha-bound}.

Let \(P^\alpha\) be the Taylor polynomial of degree \(s-1\) of \(f\) at the center of \(C_\alpha\), and set
\(R^\alpha:=f-P^\alpha\). Standard Taylor estimates on \(\widehat C_\alpha\) give
\begin{equation}\label{eq:remainder-sup-bound}
\|R^\alpha\|_{L_\infty(\widehat C_\alpha)}
\le C_T(d,s)\,\widetilde h^s
\max_{|\beta|=s}\|D^\beta f\|_{L_\infty(\widehat C_\alpha)},
\end{equation}
\begin{equation}\label{eq:remainder-grad-bound}
\|\nabla R^\alpha\|_{L_\infty(\widehat C_\alpha)}
\le C_T(d,s)\,\widetilde h^{s-1}
\max_{|\beta|=s}\|D^\beta f\|_{L_\infty(\widehat C_\alpha)}.
\end{equation}
Because both cubature rules are exact on \(P^\alpha\),
\[
\left|A^\alpha(f)-\widetilde A^{\alpha,Y}(f)\right|
\le \bigl(1+\|\lambda^{\alpha,Y}\|_{\ell_1}\bigr)\|R^\alpha\|_{L_\infty(\widehat C_\alpha)},
\]
\[
\left|A^\alpha(f)-\widetilde A^{\alpha,Q}(f)\right|
\le \bigl(1+\|\lambda^{\alpha,Q}\|_{\ell_1}\bigr)\|R^\alpha\|_{L_\infty(\widehat C_\alpha)},
\]
which, together with \eqref{eq:Y-weight-l1} and \eqref{eq:remainder-sup-bound}, proves
\eqref{eq:Y-cubature-error} and \eqref{eq:Q-cubature-error}.

For the difference of the two cubature functionals, exactness on \(P^\alpha\) gives
\begin{flalign*}
&\widetilde A^{\alpha,Q}(f)-\widetilde A^{\alpha,Y}(f)  \\
&=
\sum_{\nu=1}^M(\lambda^{\alpha,Q}_\nu-\lambda^{\alpha,Y}_\nu)R^\alpha(q_{\alpha,\nu})
+
\sum_{\nu=1}^M\lambda^{\alpha,Y}_\nu\bigl(R^\alpha(q_{\alpha,\nu})-R^\alpha(y_{\alpha,\nu})\bigr).
\end{flalign*}

Therefore, by \eqref{eq:YQ-weight-difference}, \eqref{eq:Y-weight-l1}, \eqref{eq:remainder-sup-bound}, and
\eqref{eq:remainder-grad-bound},
\begin{flalign*}
&\left|\widetilde A^{\alpha,Q}(f)-\widetilde A^{\alpha,Y}(f)\right|\\
&\le
\|\lambda^{\alpha,Q}-\lambda^{\alpha,Y}\|_{\ell_1}\,\|R^\alpha\|_{L_\infty(\widehat C_\alpha)}
+ \|\lambda^{\alpha,Y}\|_{\ell_1}\,\|\nabla R^\alpha\|_{L_\infty(\widehat C_\alpha)}\,(2r)\\
&\le C(d,s,\eta_0)\,r\,\widetilde h^{s-1}
\max_{|\beta|=s}\|D^\beta f\|_{L_\infty(\widehat C_\alpha)},
\end{flalign*}
which is \eqref{eq:YQ-cubature-difference}.
\end{proof}

The same construction applies with any other explicit lower bound ensuring the invertibility of the reference Vandermonde
matrix, for example after selecting a reference skeleton by a metric-span criterion.

\begin{remark}
\label{rem:scale-sep}
Condition \eqref{eq:scale-separation} is an upper bound on scale separation $r/h$. All other constants are scale-independent, so
the largest admissible $r$ is a constant multiple of $h$ with a constant depending only on $d, s$, and $\eta_0$. Any fixed $\kappa\in (0, 1)$ satisfying $\kappa < 2MM!(s-1)/\eta_0$ works at the expense of inflating the norm of $(V^{\alpha,Q})^{-1}$ by a factor
of $(1-\kappa)^{-1}$. 
\end{remark}

\section{Translation-invariant stencil quasi-interpolants on an anchored grid}\label{sec:stencil}

This section reviews a translation-invariant quasi-interpolation framework that guarantees a standard sufficient condition for polynomial reproduction and approximation order $s$
for quasi-interpolants built from inflated-cell averages on a fixed cubic grid.

\subsection{Polynomial reproduction and approximation order}

Let $\phi$ be a compactly supported generator on $\mathbb{R}^d$. For $h>0$, we consider quasi-interpolants of the form
\[
(Q_h f)(x):=\sum_{\alpha\in\mathbb{Z}^d} c^\alpha(f)\,\phi_h(x-h\alpha), \qquad \phi_h(x):=(\sigma_h\phi)(x) := \phi(x/h).
\]
Standard Strang-Fix theory shows that, under appropriate compatibility conditions, polynomial reproduction of degree $s-1$
implies approximation order $s$; see \cite{StrangFix73,Jia2004,LeiJia1997}.
For later use we record the convolutional formulation employed by Lei-Jia-Cheney.

\begin{theorem}\label{thm:reproduction-besov}
Let $s\in\mathbb N$, $1\le p\le\infty$, and let $p'$ be the
Hölder conjugate of $p$. Let
$\phi\in L_p(\mathbb R^d)\cap L_1(\mathbb R^d),
g\in L_{p'}(\mathbb R^d)\cap L_1(\mathbb R^d)$,
with $g$ compactly supported. Define
\begin{equation}\label{eq:P-def}
(Pf)(x)
=
\sum_{\alpha\in\mathbb Z^d}
(f*g)(\alpha)\phi(x-\alpha),
\qquad 
Q_h:=\sigma_hP\sigma_{1/h}.
\end{equation}
Let
\begin{equation}\label{eq:K-def}
K(x,y):=\sum_{\alpha\in\mathbb Z^d}
\phi(x-\alpha)g(\alpha-y)
\end{equation}
be the kernel associated with $P$. Assume that $K$ satisfies
\begin{equation}\label{eq:K-shift}
K(x-\nu,y)=K(x,y+\nu),
\qquad \nu\in\mathbb Z^d, \; \text{for a.e. }x,y\in\mathbb{R}^d
\end{equation}
and the kernel bounds
\begin{equation}\label{eq:K-adm1}
y\mapsto \int_{\mathbb R^d}|K(x,y)|\,dx
\in L_\infty([0,1)^d),
\end{equation}
\begin{equation}\label{eq:K-adm2}
x\mapsto
\int_{\mathbb R^d}(1+\|y\|_\infty)^s|K(x,y)|\,dy
\in L_\infty([0,1)^d).
\end{equation}
Assume also that $\widehat\phi(0)\ne0$, that $\phi$ satisfies the
Strang-Fix conditions
\begin{equation}\label{eq:SF}
D^\beta\widehat\phi(2\pi\nu)=0,
\qquad
\nu\in\mathbb Z^d\setminus\{0\},
\quad |\beta|\le s-1,
\end{equation}
and that $\phi, g$ satisfy the coupling conditions
\begin{equation}\label{eq:compat}
D^\beta\!\left(1-\widehat\phi\,\widehat g\right)(0)=0,
\qquad
|\beta|\le s-1.
\end{equation}
Then
\[
Pq=q,
\qquad q\in\Pi_{s-1}(\mathbb R^d).
\]
Consequently,
\[
Q_hq=q,
\qquad q\in\Pi_{s-1}(\mathbb R^d),
\quad h>0.
\]
Moreover, for every $f\in W_p^s(\mathbb R^d)$,
\begin{equation}\label{eq:LH-Sobolev-order}
\|f-Q_hf\|_{L_p(\mathbb R^d)}
\le
C h^s |f|_{W_p^s(\mathbb R^d)},
\end{equation}
where $C$ is independent of $f$, $h$, and $p$.
\end{theorem}

These results are already known from the literature, so we provide an outline of the proof with references to the relevant existing work.

\begin{proof}[Proof outline]
The stated Strang-Fix and coupling conditions imply polynomial
reproduction for the quasi-projection $P$; see
\cite[Lem.~3.2]{Jia2003} and \cite[\S3]{Jia2004}.
If $q\in\Pi_{s-1}$, then $\sigma_{1/h}q\in\Pi_{s-1}$, so
$Q_hq
=
\sigma_hP\sigma_{1/h}q
=
\sigma_h\sigma_{1/h}q
=
q.
$
Statement \eqref{eq:LH-Sobolev-order} is \cite[Thm.~2.1]{LeiJia1997}, applied under
\eqref{eq:K-shift}--\eqref{eq:K-adm2}.
\end{proof}

\begin{lemma}\label{lem:stencil-weights}
Let
\[
M:=\dim\Pi_{s-1}(\mathbb R^d)
=
\binom{s-1+d}{d}.
\]
Let $S\subset\mathbb Z^d$ consist of $M$ lattice points and assume
that $S$ is unisolvent for $\Pi_{s-1}(\mathbb R^d)$. Fix
$\rho\ge0$, and define
\[
B_\rho(x)
=
(1+2\rho)^{-d}
\mathbf 1_{[-1/2-\rho,\;1/2+\rho]^d}(x).
\]
Let $\phi$ satisfy $\widehat\phi(0)\ne0$. Then there exist unique
coefficients $(a_j(\rho))_{j\in S}$ such that, with
\begin{equation} \label{g-rho}
\psi_\rho(x)=\sum_{j\in S}a_j(\rho)B_\rho(x-j),
\qquad
g_\rho(x):=\check\psi_{\rho}(x):=\psi_{\rho}(-x),
\end{equation}
the coupling conditions \eqref{eq:compat} hold for $\phi, g_\rho$.
Moreover, if $0\le\rho\le\rho_0<\infty$, then
\begin{equation}\label{eq:uniforml1}
\sup_{0\le\rho\le\rho_0}
\sum_{j\in S}|a_j(\rho)|<\infty.
\end{equation}
\end{lemma}

\begin{proof}
The coupling conditions are equivalent to
$
D^\beta(\widehat\phi\,\widehat g_\rho)(0)=\delta_{\beta,0}, |\beta|\le s-1.
$
By Leibniz's rule,
\[
D^\beta(\widehat\phi\,\widehat g_\rho)(0)
=
\sum_{\gamma\le\beta}
{\beta\choose\gamma}
D^\gamma\widehat\phi(0)\,
D^{\beta-\gamma}\widehat g_\rho(0).
\]
The term with $\gamma=0$ is $\widehat\phi(0)D^\beta\widehat g_\rho(0).$ All other terms contain derivatives $D^\eta\widehat g_\rho(0)$ with
$|\eta|<|\beta|$. Since $\widehat\phi(0)\ne0$, the equations can be
solved recursively in increasing order of $|\beta|$. Thus, the coupling
conditions uniquely prescribe the values
$
D^\beta\widehat g_\rho(0), |\beta|\le s-1,
$
and the Fourier-moment identities
\[
D^\beta\widehat g_\rho(0)
=
(-i)^{|\beta|}
\int_{\mathbb R^d}x^\beta g_\rho(x)\,dx, \qquad g_\rho(x)
=
\sum_{j\in S}a_j(\rho)B_\rho(x+j)
\]
then give the linear system
\[
\sum_{j\in S}
a_j(\rho)\mu_{\beta,j}(\rho)
=
M_\beta,
\qquad
\mu_{\beta,j}(\rho)
:=
\int_{\mathbb R^d}
x^\beta B_\rho(x+j)\,dx,
\]
where the $M_\beta$ are determined uniquely by the coupling
conditions.

The change of variables $u=x+j$ gives
\[
\mu_{\beta,j}(\rho)
=
\int_{\mathbb R^d}
(u-j)^\beta B_\rho(u)\,du.
\]
Applying the multi-index binomial theorem,
we obtain
\[
\mu_{\beta,j}(\rho)
=
\sum_{\eta\le\beta}
{\beta\choose\eta}
m_\eta(\rho)
(-j)^{\beta-\eta},
\qquad
m_\eta(\rho)
=
\int_{\mathbb R^d}
u^\eta B_\rho(u)\,du.
\]
Since $B_\rho$ is normalized, $m_0(\rho)=1$. Therefore
\[
\mu_{\beta,j}(\rho)
=
(-j)^\beta
+
\sum_{\substack{\gamma\le \beta\\ |\gamma|<|\beta|}}
c_{\beta,\gamma}(\rho)(-j)^\gamma ,
\]
where the sum is over multi-indices of strictly smaller total degree and
the coefficients $c_{\beta,\gamma}(\rho)$ depend continuously on
$\rho$.

Order the multi-indices by nondecreasing total degree and let
\[
V=\bigl((-j)^\beta\bigr)_{|\beta|\le s-1,\;j\in S}
\]
be the multivariate Vandermonde matrix associated with the stencil
$S$. The preceding expansion shows that $\mu(\rho)=L(\rho)V$,
where $L(\rho)$ is unit lower triangular. Consequently,
$
\det\mu(\rho)
=
\det L(\rho)\det V
=
\det V.
$
Since $S$ is unisolvent for
$\Pi_{s-1}(\mathbb R^d)$, the Vandermonde matrix $V$ is
nonsingular. Hence $\mu(\rho)$ is nonsingular for every
$\rho\ge0$, and the coefficients $a_j(\rho)$ exist uniquely.

Every entry of $\mu(\rho)$ is a polynomial in
$\rho$, so $\mu(\rho)$ depends continuously on
$\rho$ and $\mu(\rho)^{-1}$ is continuous on every
compact interval $[0,\rho_0]$. Since the moment vector
$M=(M_\beta)$ is independent of $\rho$, the solution $a(\rho) = \mu(\rho)^{-1}M$ depends continuously on $\rho$. This establishes \eqref{eq:uniforml1}.
\end{proof}

As a consequence of \eqref{eq:uniforml1} and each $g_\rho$ having compact support with $||g_\rho||_{L_p(\mathbb R^d)}$ uniformly bounded for $0\le \rho\le \rho_0<\infty$ for any $1\le p \le \infty$, provided also
\[
B_\phi:=
\operatorname*{ess\,sup}_{x\in[0,1)^d}
\sum_{\alpha\in\mathbb Z^d}|\phi(x-\alpha)|<\infty,
\]
the kernels
$
K_\rho(x,y)
:=
\sum_{\alpha\in\mathbb Z^d}
\phi(x-\alpha)g_\rho(\alpha-y)
$
satisfy \eqref{eq:K-shift}, \eqref{eq:K-adm1}, \eqref{eq:K-adm2} uniformly in $\rho$. Then the methods in \cite{StrangFix73,Jia2004,LeiJia1997} used to establish Theorem \ref{thm:reproduction-besov} similarly prove the following.

\begin{corollary}\label{cor:reproduction-besov-scaled}
Assume that $\phi\in L_p(\mathbb R^d)\cap L_1(\mathbb R^d)$, $1\le p\le\infty$, is compactly supported, $B_\phi < \infty$, $\widehat\phi(0)\ne0$, and $\phi$ satisfies the Strang-Fix conditions \eqref{eq:SF}. Let $0\le\rho(h)\le\rho_0<\infty$. Then $g_{\rho(h)}$ may be defined as in Lemma \ref{lem:stencil-weights} so that the stencil weights $(a_j(\rho))_{j\in S}$ are unique and $g_\rho$ satisfies coupling conditions \eqref{eq:compat}. The quasi-projection operator
\[
P_{\rho}f
=
\sum_{\alpha\in\mathbb Z^d}
(f*g_\rho)(\alpha)\phi(\cdot-\alpha),
\qquad 
Q_h
=
\sigma_hP_{\rho(h)}\sigma_{1/h} 
\]
satisfies
\[
Q_hq=q,
\qquad q\in\Pi_{s-1}(\mathbb R^d),
\]
and
\[
\|f-Q_hf\|_{L_p(\mathbb R^d)}
\le
C h^s |f|_{W_p^s(\mathbb R^d)},
\qquad f\in W_p^s(\mathbb R^d),
\]
where $C$ is independent of $f$, $p$, $r$, and $h$. 
\end{corollary}

\subsection{Grid, local averages, and translation-invariant stencils}

Let $h\mathbb{Z}^d$ be the cubic lattice on $\mathbb{R}^d$ (or on a torus to avoid boundaries).

For each $\alpha\in\mathbb{Z}^d$, let
\[
C_\alpha:=h\alpha+[-h/2,h/2]^d,
\quad
\widehat C_\alpha:=\Thick(C_\alpha):=h\alpha+[-h/2-r,h/2+r]^d.
\]

Define the inflated local average
$
A^\alpha(f):=(\mathcal{L}^d(\widehat C_\alpha))^{-1}\int_{\widehat C_\alpha} f(x)\,dx.
$
In addition, for $\rho=r/h$ define
\[
c_{(h)}^\alpha(f)
:=
\int_{\mathbb R^d}
f(y)g_{h,r}(h\alpha-y)\,dy,
\qquad 
g_{h,r}(x):=h^{-d}g_\rho(x/h),
\]
with $g_\rho$ as in Lemma \ref{lem:stencil-weights}. Applying the same lemma to identify coefficients $(a_j(\rho))_{j\in S}$ and using a change of variables then shows 
\[
\int_{\mathbb R^d}
f(y)g_{h,r}(h\alpha-y)\,dy = \sum_{j\in S}a_j(\rho)A^{\alpha+j}(f).
\]
Thus, the coefficient functionals $c_{(h)}^\alpha(f)$ depend only on integrals of $f$ over inflated cubes.

\begin{remark}
Corollary~\ref{cor:reproduction-besov-scaled} is stated on $\mathbb{R}^d$, while in the present paper the target function is given on the bounded cube $U$. The geometric scales are different but compatible: the coarse partition uses cubes $C_\alpha$ of side length $h$, the local integral functionals are taken over the inflated cubes $\widehat C_\alpha=C_\alpha+[-r,r]^d$, and the reconstruction itself is performed on the same coarse lattice $h\mathbb{Z}^d$. Thus, the periodic quasi-interpolant is global only in the sense that its translates are indexed by the full coarse grid; its local functionals still probe $f$ only on an $O(h)$ neighborhood of each coarse cube, since $\widetilde h=h+2r\asymp h$ in the regime $r<h$.

For the extension-restriction step, one may first extend $f$ from $U$ to a function $Ef$ on $\mathbb{R}^d$ (or, equivalently, to an $O(h)$ neighborhood surrounding $U$ that contains all inflated cubes required by the stencil), then apply the global anchored-grid quasi-interpolant to $Ef$, and finally restrict the result back to $U$. Because every coefficient functional is supported on finitely many inflated cubes, only values of $Ef$ in a neighborhood of thickness $O(h)$ around $U$ are used. Hence the approximation scale remains the coarse-grid scale $h$. If a bounded extension operator is available on the underlying smoothness space, its norm enters only into the constant, while the order in $h$ is unchanged. On periodic domains one may avoid the extension entirely by working on a torus from the outset.
\end{remark}

\section{Cubature-based high-order quasi-interpolation on non-quasi-uniform node sets}\label{sec:final-cubature-qi}

We now combine the local cubature results of Section~\ref{sec:local-cubature} with the translation-invariant stencil
framework of Section~\ref{sec:stencil}. The quasi-interpolant requires exact integral averages, but these averages are
not available. Since we wish to approximate functions using disordered node sets $Q$, a direct approach would be to
replace the integrals by a $Q$-cubature and then estimate the induced error. This is not an available route, because
$Q$ itself is only partially known. Therefore, we replace the unavailable $Q$-cubatures by the known cubatures associated
with $Y$. Thus $Y$ is used as a controlled perturbation of $Q$ for the purpose of constructing an explicit quasi-interpolant. The resulting operator is indexed by $Q$, while its computable cubature part is $Y$-based.

\subsection{Assumptions for the perturbative theorem}\label{subsec:QI-assumptions}

We collect the hypotheses used in Theorem~\ref{thm:QI-perturbed-reference}.  They are separated into the assumptions
belonging to the quasi-interpolant and those needed for the cubature replacement.

\subsubsection{Quasi-interpolation assumptions}

\begin{assumption}\label{ass:QI}
Let \(\phi\) be a generator on \(\mathbb{R}^d\) satisfying Strang-Fix conditions \eqref{eq:SF} and $\widehat\phi(0)\ne0$ and let \(\mathcal{S}\subset\mathbb{Z}^d\) be a stencil. For $\rho:=r/h$ with $0\le\rho\le\rho_0<\infty$, let 
$(a_j(\rho))_{j\in S}$ be coefficients. We assume the following:
\begin{enumerate}
\item[\textup{(QI1)}] The generator \(\phi\) is compactly supported and
\[
B_\phi:=\operatorname*{ess\,sup}_{x\in[0,1)^d}
\sum_{\alpha\in\mathbb Z^d}|\phi(x-\alpha)|<\infty.
\]
\item[\textup{(QI2)}] The stencil \(\mathcal{S}\) is finite and $\sup_{0\le\rho\le\rho_0}
\sum_{j\in S}|a_j(\rho)|<\infty.$
\item[\textup{(QI3)}] The translation-invariant quasi-interpolant based on inflated averages,
\begin{equation}\label{eq:continuous-QI-final}
(Q_h f)(x):=\sum_{\alpha\in\mathbb{Z}^d} c^\alpha(f)\,\phi_h(x-h\alpha),
\quad
c^\alpha(f):=\sum_{j\in\mathcal{S}} a_j(\rho) A^{\alpha+j}(f),
\end{equation}
reproduces all polynomials of degree at most \(s-1\) and satisfies
\begin{equation}\label{eq:QI-assumption-final}
\|f-Q_h f\|_{L_\infty(U)}
\le
C_{\mathrm{QI}}h^s
\max_{|\beta|=s}\|D^\beta f\|_{L_\infty(\widehat U)}
\end{equation}
for all \(f\in C^s(\widehat U)\), where \(\widehat U\) contains every inflated cube used by the stencil.
\end{enumerate}
\end{assumption}

Assumption~\ref{ass:QI} \textup{(QI3)} is the standard approximation estimate obtained from polynomial reproduction and the
Strang-Fix conditions, which is justified under the construction used in Corollary \ref{cor:reproduction-besov-scaled}.

\subsubsection{Combinatorial and cubature replacement assumptions}

For every coarse cube entering the stencil construction, let
\[
Y^\alpha\subset \Thick(C_\alpha),\qquad Q^\alpha\subset C_\alpha,
\qquad |Y^\alpha|=|Q^\alpha|,
\]
be the local reference set and the local disordered node set.  The set \(Y^\alpha\) is known.  The set \(Q^\alpha\) is the
local geometry to which the quasi-interpolant is attached.  The point is that the computable cubature will use only nodes
from \(Y^\alpha\), while the validity of this replacement is supplied by the HT matching theorem and the cubature replacement
theorem.

Let
\[
M:=\dim \Pi_{s-1}(\mathbb R^d),
\qquad
\widetilde h:=h+2r.
\]
When a reference skeleton
\[
(Y^\alpha)_{\rm sk}=\{y_{\alpha,1},\dots,y_{\alpha,M}\}\subset Y^\alpha
\]
has been selected, define the corresponding reference cubature by
\begin{equation}\label{eq:reference-Y-average}
\widetilde A^{\alpha,Y}(f):=\sum_{\nu=1}^{M}\lambda^{\alpha,Y}_\nu f(y_{\alpha,\nu}),
\end{equation}
where the coefficients \(\lambda^{\alpha,Y}_\nu\) are chosen so that the rule is exact on \(\Pi_{s-1}(\mathbb R^d)\) for the
inflated average \(A^\alpha\).  The corresponding cubature on \(Q^\alpha\) is not used in the definition below.  It enters
only in the proof, after the HT theorem supplies the matching from \(Y^\alpha\) to \(Q^\alpha\).

\begin{definition}\label{def:QY-quasi-interpolant}
For \(f\in C^s(\widehat U)\) and $\rho:=r/h>0$, define
\begin{equation}\label{eq:QY-coefficients}
b^\alpha(f):=\sum_{j\in\mathcal S}a_j(\rho) \widetilde A^{\alpha+j,Y}(f),
\end{equation}
and assign the \(Q\)-quasi-interpolant
\begin{equation}\label{eq:QY-quasi-interpolant}
(\mathcal Q_{h,Q}f)(x):=\sum_{\alpha\in\mathbb Z^d}b^\alpha(f)\phi_h(x-h\alpha).
\end{equation}
\end{definition}

The subscript \(h,Q\) emphasizes this construction is assigned to the actual disordered set $Q$, while the formula itself uses only the
known cubature rules on \(Y\).

\begin{assumption}\label{ass:cubature-perturbation}
For every coarse cube \(C_\alpha\) entering the stencil construction, the following assumptions hold uniformly in \(\alpha\):
\begin{enumerate}
\item[\textup{(C1)}] The local population vectors of \(Y^\alpha\) and \(Q^\alpha\), computed on the mesoscopic partition of
scale \(r\), satisfy the Hall-type inequalities \eqref{eq:hall-majorization} in Theorem~\ref{thm:matching}.  Equivalently, the HT hypotheses hold locally
for the pair \((Y^\alpha,Q^\alpha)\).
\item[\textup{(C2)}] The known set \(Y^\alpha\) contains a reference skeleton
\((Y^\alpha)_{\rm sk}=\{y_{\alpha,1},\dots,y_{\alpha,M}\}\) whose normalized Vandermonde determinant satisfies
\[
\bigl|\Delta_{s-1}(\overline{\overline{Y^\alpha}}_{\rm sk})\bigr|\ge \eta_0>0.
\]
The same constant \(\eta_0\) is used for all \(\alpha\).  In addition, \(0<r<h\), and the scale \(r\) satisfies the smallness condition from
Theorem~\ref{thm:transfer-reference-cubature}, namely
\[
M(s-1)\frac{M!}{\eta_0}\frac{4r}{h+2r}\le \kappa,
\]
with $\kappa \in (0,1)$ fixed such that $\kappa < 2MM!(s-1)/\eta_0$.
\end{enumerate}
\end{assumption}

Assumption \textup{(C1)} is the combinatorial input.  By Theorem~\ref{thm:matching}, it gives a matching
\(\tau^\alpha:Y^\alpha\to Q^\alpha\) with range at most \(2r\).  Assumption \textup{(C2)} is imposed only on the known set
\(Y^\alpha\), together with the scale separation required by the cubature replacement theorem.  Once the matching from
\textup{(C1)} is available, \textup{(C2)} supplies the remaining hypotheses of Theorem~\ref{thm:transfer-reference-cubature}.

\begin{theorem}\label{thm:QI-perturbed-reference}
Assume \textup{(QI1)--(QI3)} and \textup{(C1)--(C2)}. Then the \(Q\)-quasi-interpolant
\(\mathcal Q_{h,Q}\) of Definition~\ref{def:QY-quasi-interpolant} satisfies
\begin{equation}\label{eq:QI-perturbed-bound}
\begin{aligned}
\|f-\mathcal Q_{h,Q} f\|_{L_\infty(U)}
&\le
C\bigl(h^s+r h^{s-1}\bigr)
\max_{|\beta|=s}\|D^\beta f\|_{L_\infty(\widehat U)} \\
&\qquad\text{for all }f\in C^s(\widehat U),
\end{aligned}
\end{equation}
where \(C\) depends only on \(d,s\), the generator \(\phi\), the stencil \(\mathcal S\), the uniform bound on coefficients $\{a_j(\rho)\}$ given by (QI2), and
the uniform constants in Theorem~\ref{thm:transfer-reference-cubature}. Since \(r<h\), the right-hand side is of order \(h^s\).
Moreover, if \(\mathcal{H}_{h,Q}\) denotes the auxiliary quasi-interpolant obtained from the unavailable cubatures on
\(Q\), then
\begin{equation}\label{eq:QY-to-Q-comparison}
\|\mathcal{H}_{h,Q}f-\mathcal Q_{h,Q}f\|_{L_\infty(U)}
\le
C r h^{s-1}
\max_{|\beta|=s}\|D^\beta f\|_{L_\infty(\widehat U)}.
\end{equation}
\end{theorem}

\begin{proof}
Set
\begin{equation*}
D_s(f):=\max_{|\beta|=s}\|D^\beta f\|_{L_\infty(\widehat U)}.
\end{equation*}
By \textup{(C1)} and Theorem~\ref{thm:matching}, for each relevant \(\alpha\) there is a bijection
\[
\tau^\alpha:Y^\alpha\to Q^\alpha,
\qquad
\|\tau^\alpha(y)-y\|_\infty\le 2r.
\]
This is the matching range needed in Theorem~\ref{thm:transfer-reference-cubature}.  The determinant margin in
\textup{(C2)} is imposed on the reference skeleton in \(Y^\alpha\), and the scale condition in \textup{(C2)} is precisely the
smallness condition required there.  Hence, Theorem~\ref{thm:transfer-reference-cubature} applies uniformly in \(\alpha\).
It gives the cubature on \(Y^\alpha\), the auxiliary cubature on the matched skeleton in \(Q^\alpha\), and the estimates
\begin{equation}\label{eq:local-Q-cubature-error-proof}
|A^\alpha(f)-\widetilde A^{\alpha,Q}(f)|
\le C_Q\widetilde h^sD_s(f),
\end{equation}
\begin{equation}\label{eq:local-YQ-transfer-error-proof}
|\widetilde A^{\alpha,Q}(f)-\widetilde A^{\alpha,Y}(f)|
\le C_{YQ}r\widetilde h^{s-1}D_s(f).
\end{equation}
Here \(\widetilde A^{\alpha,Q}\) is used only in this proof.

Let \(c^\alpha(f)\) be the exact coefficient functional in \eqref{eq:continuous-QI-final}.  Define the auxiliary coefficients
\[
c^{\alpha,Q}(f):=\sum_{j\in\mathcal S}a_j(\rho)\widetilde A^{\alpha+j,Q}(f),
\quad
(\mathcal{H}_{h,Q}f)(x):=\sum_{\alpha\in\mathbb Z^d}c^{\alpha,Q}(f)\phi_h(x-h\alpha).
\]
By \textup{(QI2)} and \eqref{eq:local-Q-cubature-error-proof},
\begin{equation}\label{eq:exact-to-Q-coefficient-proof}
|c^\alpha(f)-c^{\alpha,Q}(f)|
\le
\Bigl(\sup_{0\le\rho\le\rho_0}\sum_{j\in\mathcal S}|a_j(\rho)|\Bigr) C_Q\widetilde h^sD_s(f).
\end{equation}
By \textup{(QI2)} and \eqref{eq:local-YQ-transfer-error-proof},
\begin{equation}\label{eq:Q-to-Y-coefficient-proof}
|c^{\alpha,Q}(f)-b^\alpha(f)|
\le
\Bigl(\sup_{0\le\rho\le\rho_0}\sum_{j\in\mathcal S}|a_j(\rho)|\Bigr) C_{YQ}r\widetilde h^{s-1}D_s(f).
\end{equation}
Using the bounded overlap from \textup{(QI1)}, \eqref{eq:exact-to-Q-coefficient-proof} gives
\begin{equation}\label{eq:ideal-to-auxiliary-Q-proof}
\|Q_hf-\mathcal{H}_{h,Q}f\|_{L_\infty(U)}
\le C_1\widetilde h^sD_s(f),
\end{equation}
and \eqref{eq:Q-to-Y-coefficient-proof} gives
\begin{equation}\label{eq:auxiliary-Q-to-Y-proof}
\|\mathcal{H}_{h,Q}f-\mathcal Q_{h,Q}f\|_{L_\infty(U)}
\le C_2r\widetilde h^{s-1}D_s(f).
\end{equation}
The second estimate is \eqref{eq:QY-to-Q-comparison} after using \(r<h\), as required in \textup{(C2)}.

Finally, \textup{(QI3)} gives
\begin{equation}\label{eq:ideal-QI-error-proof}
\|f-Q_h f\|_{L_\infty(U)}
\le
C_{\mathrm{QI}}h^sD_s(f).
\end{equation}
Combining \eqref{eq:ideal-QI-error-proof}, \eqref{eq:ideal-to-auxiliary-Q-proof}, and
\eqref{eq:auxiliary-Q-to-Y-proof}, and using \(\widetilde h=h+2r\le 3h\), gives
\eqref{eq:QI-perturbed-bound}.  This completes the proof.
\end{proof}

For Theorem~\ref{thm:QI-direct-gamma}, we will instead assume the node set $Q$ is completely known and identify geometric conditions on $Q$ that are sufficient for high-order quasi-interpolation. This leaves the perturbative comparison with $Y$ aside and is not the main focus of the paper, but this theorem clarifies the role of local cubature on $Q$ and gives a useful point of comparison with the quasi-uniform geometry discussed afterward.

\begin{theorem}\label{thm:QI-direct-gamma}
Let \(Q=\bigcup_\alpha Q^\alpha\subset U\) be an arbitrary finite node set. Assume \textup{(QI1)--(QI3)}. Assume also that,
for every coarse cube \(C_\alpha\) entering the stencil construction, the normalized local set satisfies
\begin{equation}\label{eq:gamma-Q-direct}
\Gamma_{s-1}(\overline{\overline{Q^\alpha}})\ge \gamma_0>0.
\end{equation}
Then the quasi-interpolant \(\widetilde Q_h\), obtained by replacing each inflated average by the local
cubature of Theorem~\ref{thm:local-cubature-gamma}, satisfies
\begin{equation}\label{eq:QI-direct-bound}
\begin{aligned}
\|f-\widetilde Q_h f\|_{L_\infty(U)}
&\le
C h^s\max_{|\beta|=s}\|D^\beta f\|_{L_\infty(\widehat U)} \\
&\qquad\text{for all }f\in C^s(\widehat U),
\end{aligned}
\end{equation}
where \(C\) depends only on \(d,s\), the generator \(\phi\), the stencil \(\mathcal{S}\), $\sup_{0\le\rho\le\rho_0}\sum_{j\in\mathcal S}|a_j(\rho)|$, and \(\gamma_0\).
\end{theorem}

\begin{proof}
Theorem~\ref{thm:local-cubature-gamma} yields local cubature rules on each relevant set \(Q^\alpha\), with a uniform \(\ell_1\)-bound
\eqref{eq:l1-average-bound} depending only on \(\gamma_0\). The resulting local cubature error is bounded by \eqref{eq:cubature-error}.
The proof now repeats the argument in Theorem~\ref{thm:QI-perturbed-reference}: the coefficient errors are controlled by
\textup{(QI2)}, and the final lattice sum is controlled by the bounded overlap in \textup{(QI1)}.
Adding the anchored-grid approximation estimate from \textup{(QI3)} gives \eqref{eq:QI-direct-bound}.
\end{proof}

\begin{corollary}\label{cor:QI-direct-explicit}
The conclusion of Theorem~\ref{thm:QI-direct-gamma} holds whenever, for every coarse cube entering the stencil construction,
at least one of the following local conditions is satisfied uniformly in \(\alpha\):
\begin{enumerate}
\item the normalized local set \(\overline{\overline{Q^\alpha}}\) contains a subset of cardinality \(M=\dim\Pi_{s-1}(\mathbb{R}^d)\) whose
Vandermonde determinant satisfies
\[
|\Delta_{s-1}|\ge \eta_0>0;
\]
\item the metric span satisfies
\[
\omega_{s-1,d}(\overline{\overline{Q^\alpha}})\ge \omega_0>0;
\]
\item the set is finite with minimal separation \(\varepsilon_\alpha=\operatorname{sep}_\infty(\overline{\overline{Q^\alpha}})\) and
\[
\left|\overline{\overline{Q^\alpha}}\right| > M_{s-1,d}(\varepsilon_\alpha).
\]
\end{enumerate}
\end{corollary}

\begin{proof}
In case \textup{(1)}, Proposition~\ref{prop:det-gamma} gives a uniform lower bound for \(\Gamma_{s-1}(\overline{\overline{Q^\alpha}})\).
In case \textup{(2)}, Proposition~\ref{prop:metric-span-gamma} gives such a lower bound.
In case \textup{(3)}, Corollary~\ref{cor:finite-criterion} gives such a lower bound.
The conclusion then follows from Theorem~\ref{thm:QI-direct-gamma}.
\end{proof}

\subsection{Relation with quasi-uniformity}

We briefly compare the direct hypothesis \eqref{eq:gamma-Q-direct} with the classical quasi-uniform setting.
The comparison is local.
It is enough to control the normalized fill distance of the local sample in each inflated cube; a separate lower bound on the
separation is not needed for the cubature step itself.  This shows that quasi-uniformity is sufficient for our hypothesis, but stronger than necessary.

For a finite set \(Z\subset\Qone\), define its normalized fill distance by
\[
h_\infty(Z,\Qone):=\sup_{x\in\Qone}\min_{z\in Z}\|x-z\|_\infty.
\]

\begin{proposition}\label{prop:fill-gamma}
Let \(Z\subset\Qone\) be finite. Then
\[
\Gamma_{s-1}(Z)\ge 1-\mathfrak{M}(d,s)\,h_\infty(Z,\Qone),
\]
where \(\mathfrak{M}(d,s)\) is the Markov constant from Lemma~\ref{lem:gamma-perturb}.
In particular, if
\[
h_\infty(Z,\Qone)\le \frac{1}{2\mathfrak{M}(d,s)},
\]
then
\[
\Gamma_{s-1}(Z)\ge \frac12.
\]
\end{proposition}

\begin{proof}
Let \(p\in\Pi_{s-1}(\mathbb{R}^d)\) satisfy \(\|p\|_{L_\infty(\Qone)}=1\). Choose \(x_\ast\in\Qone\) such that \(|p(x_\ast)|=1\).
By definition of \(h_\infty(Z,\Qone)\), there exists \(z\in Z\) with
\[
\|x_\ast-z\|_\infty\le h_\infty(Z,\Qone).
\]
By the Markov brothers' inequality on \(\Qone\),
\[
|p(z)-p(x_\ast)|\le \mathfrak{M}(d,s)\,h_\infty(Z,\Qone).
\]
Hence
\[
|p(z)|\ge 1-\mathfrak{M}(d,s)\,h_\infty(Z,\Qone),
\]
and therefore
\[
\max_{\zeta\in Z}|p(\zeta)|\ge 1-\mathfrak{M}(d,s)\,h_\infty(Z,\Qone).
\]
Taking the infimum over all such \(p\) proves the claim.
\end{proof}

\begin{corollary}\label{cor:quasi-uniform-implies-gamma}
Assume that, for every coarse cube entering the stencil construction, the normalized local sample \(\overline{\overline{Q^\alpha}}\subset\Qone\) satisfies
\[
h_\infty(\overline{\overline{Q^\alpha}},\Qone)\le \theta_0,
\qquad
\theta_0<\frac{1}{\mathfrak{M}(d,s)}.
\]
Then
\[
\Gamma_{s-1}(\overline{\overline{Q^\alpha}})\ge 1-\mathfrak{M}(d,s)\theta_0>0
\]
uniformly in \(\alpha\), and therefore Theorem~\ref{thm:QI-direct-gamma} applies.
In particular, this covers the standard quasi-uniform regime whenever the local nodes fill each inflated cube at scale comparable to \(h\).
\end{corollary}

\begin{proof}
This is immediate from Proposition~\ref{prop:fill-gamma}.
\end{proof}

\begin{proposition}\label{prop:not-converse-QU}
For every \(s\) and \(d\), there exist families of finite samples \(Q\subset U\) which satisfy the hypothesis
\eqref{eq:gamma-Q-direct} of Theorem~\ref{thm:QI-direct-gamma} with one uniform constant \(\gamma_0>0\), but are not quasi-uniform.
\end{proposition}

\begin{proof}
Fix, in each coarse cube \(C_\alpha\), a set \(Y^\alpha\subset C_\alpha\) such that the normalized set \(S^\alpha(Y^\alpha)\subset\Qone\) has
\[
\Gamma_{s-1}(S^\alpha(Y^\alpha))\ge \gamma_\ast>0.
\]
For instance, one may take the same well-conditioned reference pattern in every coarse cube.
Now define
\[
Q^\alpha:=Y^\alpha\cup G^\alpha,
\]
where \(G^\alpha\subset C_\alpha\) is any additional finite set. Since $Y^\alpha$ is a well-conditioned skeleton for $Q^\alpha$, 
\[
\Gamma_{s-1}(\overline{\overline{Q^\alpha}})
\ge
\Gamma_{s-1}(S^\alpha(Y^\alpha))
\ge \gamma_\ast.
\]
Thus the direct hypothesis \eqref{eq:gamma-Q-direct} holds with \(\gamma_0=\gamma_\ast\), independently of the choice of the extra points.

Choose now one coarse cube and sequentially insert into it two distinct points whose mutual distance tends to zero.
Then the global separation distance of \(Q\) tends to zero, while the fill distance of the whole sample in \(U\) remains bounded by the original coarse scale.
Hence, the quasi-uniformity ratio becomes arbitrarily large as $N\to\infty$.
Therefore these samples are not quasi-uniform, although the local lower bound for \(\Gamma_{s-1}\) remains unchanged.
\end{proof}

\begin{remark}\label{rem:wendland-viewpoint}
Theorem~\ref{thm:QI-direct-gamma} and Corollary~\ref{cor:QI-direct-explicit} do not impose a global lower bound on the separation distance. The metric-span and finite separation-cardinality conditions provide a non-perturbative replacement for quasi-uniformity. This must be interpreted correctly. In general, there are two separate length scales: one is defined by the fill distance, while the other one is associated with pairwise minimal separation distance. A quasi-uniformity assumption requires that these scales are commensurate uniformly in $N$. The order of the resulting estimates is given as a power of the fill distance. This is true for the present approach as well. The contrast between the two methods is that in the quasi-uniform setting the fill distance is vanishing as $N\to\infty$, with the rate determined by a "microscale" associated with the separation. In the present approach, the fill distance may be on the order of $r$, which may be as large as $\theta h$ for $\theta\in (0, 1)$.  The minimal separation is not restricted, at least not on the entire sample. The allowed geometries are reasonably general, but the estimates are written in terms of a larger scale parameter $h \sim r$. In contrast, in the quasi-uniform geometry, the size of $h$ is on the order of the minimal separation. 
\end{remark}

\section{Conclusion}
Given a disordered node set $Q$, we showed that a relatively small family of inequalities based on ordered mesoscopic cube populations provide sufficient conditions for the existence of an \(O(r)\)-range perfect matching from $Q$ to a reference node set $Y$. Using the matching, we established stability of cubature based on $Q$ provided the cubature based on $Y$ is stable. Combined with translation-invariant quasi-interpolation on a coarse grid, these results give high-order approximation estimates under the scale separation condition \(r\le C h\), where the coarse grid scale $h$ determines the approximation order and the mesoscopic scale $r$ controls the matching range and size of the inflated cubes.

\section*{Acknowledgment}
The work of Alexander Panchenko was supported in part by a grant from W. M. Keck Foundation.


\newpage
\bibliographystyle{amsplain}
\bibliography{references}

\providecommand{\bysame}{\leavevmode\hbox to3em{\hrulefill}\thinspace}
\providecommand{\MR}{\relax\ifhmode\unskip\space\fi MR }
\providecommand{\MRhref}[2]{%
  \href{http://www.ams.org/mathscinet-getitem?mr=#1}{#2}
}
\providecommand{\href}[2]{#2}
\begin{thebibliography}{10}

\bibitem{BOLLOBAS199147}
B.~Bollobás and I.~Leader, \emph{Compressions and isoperimetric inequalities},
  Journal of Combinatorial Theory, Series A \textbf{56} (1991), no.~1, 47--62.

\bibitem{BrudnyiYomdin2016}
A.~Brudnyi and Y.~Yomdin, \emph{Norming sets and related {Remez}-type
  inequalities}, J. Aust. Math. Soc. \textbf{100} (2016), 163--181.

\bibitem{ChuiDiamond1990}
C.~K. Chui and H.~Diamond, \emph{A characterization of multivariate
  quasi-interpolation formulas and its applications}, Numerische Mathematik
  \textbf{57} (1990), 105--121.

\bibitem{DeBoorFix73}
C.~de~Boor and G.~J. Fix, \emph{Spline approximation by quasiinterpolants}, J.\
  Approx.\ Theory \textbf{8} (1973), no.~1, 19--45.

\bibitem{DickKuoSloan2013}
J.~Dick, F.~Y. Kuo, and I.~H. Sloan, \emph{High-dimensional integration: The
  quasi-{Monte Carlo} way}, Acta Numerica \textbf{22} (2013), 133--288.

\bibitem{DickPillichshammer2010}
J.~Dick and F.~Pillichshammer, \emph{Digital nets and sequences: Discrepancy
  theory and quasi--monte carlo integration}, Cambridge University Press, 2014.

\bibitem{Dudley1968}
R.~M. Dudley, \emph{Distances of probability measures and random variables},
  The Annals of Mathematical Statistics \textbf{39} (1968), no.~5, 1563--1572.

\bibitem{Faure1982}
H.~Faure, \emph{Discr{\'e}pance de suites associ{\'e}es {\`a} un syst{\`e}me de
  num{\'e}ration (en dimension {$s$})}, Acta Arithmetica \textbf{41} (1982),
  no.~4, 337--351.

\bibitem{Hall1935}
P.~Hall, \emph{On representatives of subsets}, Journal of the London
  Mathematical Society \textbf{s1-10} (1935), no.~1, 26--30.

\bibitem{HardyLittlewoodPolya}
G.~H. Hardy, J.~E. Littlewood, and G.~P\'olya, \emph{Inequalities}, 2 ed.,
  Cambridge University Press, Cambridge, 1952.

\bibitem{Jia2003}
R.-Q. Jia, \emph{Convergence rates of cascade algorithms}, Proc.\ Amer.\ Math.\
  Soc. \textbf{131} (2003), no.~6, 1739--1749.

\bibitem{Jia2004}
\bysame, \emph{Approximation with scaled shift-invariant spaces by means of
  quasi-projection operators}, J.\ Approx.\ Theory \textbf{131} (2004), no.~1,
  30--46.

\bibitem{Jia2010}
\bysame, \emph{Approximation by quasi-projection operators in {Besov} spaces},
  J.\ Approx.\ Theory \textbf{162} (2010), no.~1, 186--200.

\bibitem{LeiJia1997}
J.~Lei, R.-Q. Jia, and E.~W. Cheney, \emph{Approximation from shift-invariant
  spaces by integral operators}, SIAM J.\ Math.\ Anal. \textbf{28} (1997),
  no.~2, 481--498.

\bibitem{LevRudnev2018}
V.~Lev and M.~Rudnev, \emph{Minimising the sum of projections of a finite set},
  Discrete \& Computational Geometry \textbf{60} (2018), no.~2, 493--511.

\bibitem{loomiswhitney1949}
L.~H. Loomis and H.~Whitney, \emph{An inequality related to the isoperimetric
  inequality}, Bulletin of the American Mathematical Society \textbf{55}
  (1949), 961--962.

\bibitem{LycheSchumaker1975}
T.~Lyche and L.~L. Schumaker, \emph{Local spline approximation methods},
  Journal of Approximation Theory \textbf{15} (1975), no.~4, 294--325.

\bibitem{MarshallOlkinArnold}
A.~W. Marshall, I.~Olkin, and B.~C. Arnold, \emph{Inequalities: Theory of
  majorization and its applications}, 2 ed., Springer Series in Statistics,
  Springer, New York, 2011.

\bibitem{Matousek1999}
J.~Matou{\v{s}}ek, \emph{Geometric discrepancy: An illustrated guide},
  Springer, Berlin, 1999.

\bibitem{Santambrogio2015}
F.~Santambrogio, \emph{Optimal transport for applied mathematicians: Calculus
  of variations, {PDE}s, and modeling}, Birkh{\"a}user, Basel, 2015.

\bibitem{Sobol1967}
I.~M. Sobol', \emph{On the distribution of points in a cube and the approximate
  evaluation of integrals}, USSR Computational Mathematics and Mathematical
  Physics \textbf{7} (1967), no.~4, 86--112.

\bibitem{StrangFix73}
G.~Strang and G.~Fix, \emph{A {Fourier} analysis of the finite element
  variational method}, pp.~793--840, Springer Berlin, Heidelberg, 2011.

\bibitem{Strassen1965}
V.~Strassen, \emph{{The Existence of Probability Measures with Given
  Marginals}}, The Annals of Mathematical Statistics \textbf{36} (1965), no.~2,
  423 -- 439.

\bibitem{veomett2012}
E.~Veomett and A.~J. Radcliffe, \emph{Vertex isoperimetric inequalities for a
  family of graphs on ${Z}^k$}, The Electronic Journal of Combinatorics
  \textbf{19} (2012), no.~2.

\bibitem{Villani2009}
C.~Villani, \emph{Optimal transport: Old and new}, Grundlehren der
  mathematischen Wissenschaften, Springer Berlin Heidelberg, 2008.

\bibitem{NiederreiterXing1995}
C.~Xing and H.~Niederreiter, \emph{A construction of low-discrepancy sequences
  using global function fields}, Acta Arithmetica \textbf{73} (1995), no.~1,
  87--102.

\bibitem{Yomdin2011}
Y.~Yomdin, \emph{Remez-type inequality for discrete sets}, Israel J. Math.
  \textbf{186} (2011), no.~1, 45--60.

\end{thebibliography}

\end{document}